\documentclass[review,onefignum,onetabnum]{siamart190516}
\newtheorem{example}{Example}[section]
\usepackage[shortlabels]{enumitem}
\pdfoutput=1
\usepackage{amssymb}

\usepackage{lipsum}
\usepackage{amsfonts}
\usepackage{graphicx}
\usepackage{epstopdf}
\usepackage{algorithmic}
\ifpdf
\DeclareGraphicsExtensions{.eps,.pdf,.png,.jpg}
\else
\DeclareGraphicsExtensions{.eps}
\fi

\newsiamremark{remark}{Remark}
\newsiamremark{hypothesis}{Hypothesis}
\crefname{hypothesis}{Hypothesis}{Hypotheses}
\newsiamthm{claim}{Claim}

\headers{HIGH-ORDER DLQR and ILQR METHODS}{Zuodi Xie and Tieqiang Gang}

\title{DISCRETE LQR and ILQR METHODS BASED ON HIGH ORDER RUNGE-KUTTA METHODS}

\author{Zuodi Xie\thanks{Xiamen University, Xiamen 
		(\email{jordyxie@stu.xmu.edu.cn})}
	\and Tieqiang Gang\thanks{Corresponding author. Xiamen University, Xiamen (\email{gangtq@xmu.edu.cn})}
}

\usepackage{amsopn}

\ifpdf
\hypersetup{
  pdftitle={DISCRETE LQR and ILQR METHODS BASED ON HIGH ORDER RUNGE-KUTTA METHODS},
  pdfauthor={Zuodi Xie and Tieqiang Gang}
}
\fi

\begin{document}
\maketitle

\begin{abstract}
In this paper, discrete linear quadratic regulator (DLQR) and iterative linear quadratic regulator (ILQR) methods based on high-order Runge-Kutta (RK) discretization are proposed for solving linear and nonlinear quadratic optimal control problems respectively. As discovered in [W. Hager, \textit{Runge-Kutta method in optimal control and the discrete adjoint system,} Numer. Math., 2000, pp. 247-282], direct approach with RK discretization is equivalent with indirect approach based on symplectic partitioned Runge-Kutta (SPRK) integration. In this paper, we will reconstruct this equivalence by the analogue of continuous and discrete dynamic programming. Then, based on the equivalence, we discuss the issue that the internal-stage controls produced by direct approach may have lower order accuracy than the RK method used. We propose order conditions for internal-stage controls and then demonstrate that third or fourth order explicit RK discretization cannot avoid the order reduction phenomenon. To overcome this obstacle, we calculate node control instead of internal-stage controls in DLQR and ILQR methods. And numerical examples will illustrate the validity of our methods. Another advantage of our methods is high computational efficiency which comes from the usage of feedback technique. In this paper, we also demonstrate that ILQR is essentially a quasi-Newton method with linear convergence rate. 
\end{abstract}



\section{Introduction}
The numerical approaches for solving optimal control problem (OCP) are generally divided into two branches: one is indirect approach which aims at solving the canonical equations derived from Pontryagin Maximum Principle (PMP); the other is direct approach which suggests discretizing the state equations and cost function directly and transform the original problem into a finite-dimensional optimization problem \cite{Betts1998SurveyON}. In this paper, we restrict the cost function as quadratic. In the framework of direct approach, this paper aims at two tasks:
\par
\textbf{Task (1)}: How to make sure the calculated control values have the same accuracy order as the RK discretization?
\par
\textbf{Task (2)}: How to employ linear feedback technique (specifically, DLQR and ILQR methods) to solve the optimization problem after high-order RK discretization? 
\par
During the last years, much work has emerged in studying the convergence properties of RK discretization in OCP (\cite{bonnans2006computation,Dontchev2000SecondOrderRA,Hager2000RungeKuttaMI,reddien1979collocation,2012Consistent}). Though the convergence properties of RK methods applied in ordinary differential equations have been studied deeply, the analysis in OCP is usually more complicated. As exemplified in \cite{Dontchev2000SecondOrderRA,Hager2000RungeKuttaMI}, some RK discretization may not converge when solving OCP. In addition, when control or state is under constraints, stricter algebraic conditions must be satisfied by RK methods to ensure convergence or high accuracy (related work can be seen in \cite{Dontchev2000SecondOrderRA,2012Consistent}). In the non-constraint case, there exists an equivalence between direct approach based on RK discretization and indirect approach with corresponding SPRK integration \cite{Dontchev2000SecondOrderRA,SanzSerna2016SymplecticRS}. Based on this fact, Hager \cite{Hager2000RungeKuttaMI} proposed the notation of RK method’s order in OCP and computed the order conditions up to four. Then, Bonnans and Laurent-Varin \cite{bonnans2006computation} extended the order conditions to 7. Under the smooth and coercivity assumptions, a convergence order analysis for RK method in OCP is given in \cite{Hager2000RungeKuttaMI}.
\par
When applying a $s$-stage RK discretization in direct approach, one will obtain $s$ internal-stage controls. The main obstacle to increase the accuracy by utilizing high-order RK method, as discovered in \cite{Hager2000RungeKuttaMI}, is that order reduction phenomenon may happen in these internal control values. Though this problem has been referred in some papers (see \cite{albi2019linear,almuslimani2021explicit,sandu2007consistency}), deep analysis is still lack as we know. In this paper, we will give the convergence analysis of internal-stage controls and
highlight calculating node control to overcome this obstacle.
\par
After RK discretization, the OCP with quadratic objective function and linear state equations will be transformed into a discrete linear quadratic problem which can be solved by the linear feedback technique called DLQR method in this paper. When the state equations in OCP and the corresponding constraints in optimization problem are nonlinear, ILQR method, first proposed by Li and Todorov \cite{Li2004IterativeLQ}, can be employed. As a simple version of differential dynamic programming (DDP) method \cite{Mayne1970Differential}, ILQR uses iterative linearization of nonlinear system and searches for control modifications by solving discrete linear quadratic regulator problems \cite{Li2004IterativeLQ}. The prominent advantage of ILQR over other optimization technique such as sequential quadratic programming (SQP) \cite{Han1976SuperlinearlyCV} is that the search direction for iteration is obtained by the linear feedback technique, which saves the computational cost \cite{xie2017differential}. Though ILQR method cannot achieve quadratic convergence as DDP, the formal only requires calculation of the first order derivative of dynamical system and the matrix operation in which takes less computational payment \cite{giftthaler2018family}. As an efficient method to solve discrete optimal control problem, ILQR has been successfully applied in the fields of robotic and stochastic control \cite{erez2012infinite,Li2007IterativeLM,Todorov2005AGI}.\par
As remarked in \cite{murray1984differential}, technique like ILQR, in which state equations are linearized, cannot attain quadratic convergence. In \cite{giftthaler2018family}, the convergence rate of ILQR is claimed as linear but a detailed analysis is lack there. 
\subsection{Overview and notations}
The rest of this paper is organized as following.
\par
In Section \ref{section2}, we will reconstruct the equivalence between direct and indirect approaches within a dynamic programming viewpoint and discuss an important prerequisite to ensure this equivalence. In Section \ref{section3}, we deduce order conditions for internal-stage controls and reveal that some common RK discretization cannot avoid the order reduction problem. In Section \ref{section4}, we propose DLQR method based on high-order RK discretization for solving the linear quadratic OCP. Section \ref{section5} aims at solving the nonlinear quadratic OCP with ILQR method and gives a convergence analysis for the iterative progress of ILQR. The numerical examples will be given in Sections \ref{section4} and \ref{section5} to show the advantages of our methods. 
\par
In this paper, $x\in \mathbb{R}^n$ is regarded as column vector and equipped with the standard Euclidean norm $||x||=(x^Tx)^{1/2} $. For several vectors $x_1,...,x_m$, we use the notation $(x_1,...,x_m)=[x_1^T,...,x_m^T]^T$ as the composed vector. The $n$-dimension identity matrix is denoted as $I_n$. We introduce the following notation 
\begin{equation}\nonumber
	\begin{bmatrix}
		A(1) &  &0 \\
		& \ddots  & \\
		0 &   &A(m)
	\end{bmatrix}
\end{equation}
to signify the diagonal of a quasi-diagonal matrix composed of $m\in \mathbb{R}$ identical matrices $A$.
\par
Let $n_i$ be any positive integer $(i=1,2,3)$. For a differentiable function with two variables $F:\mathbb{R}^{n_1} \times \mathbb{R}^{n_2} \to \mathbb{R}^{n_3}$, $D_xF(x,y)$ (or $D_yF(x,y)$) denotes the partial derivative. The transpose of $D_xF(x,y)$ is denoted as $D_x^TF(x,y)$. For a differentiable function with single variable $G: \mathbb{R}^{n_1} \to \mathbb{R}^{n_2}$, $G'(x)$ and $G''(x)$ denote the first and second order derivatives respectively. Specially, when $n_2=1$, $G'(x)$ and $G''(x)$ signify the gradient and Hesse matrix of $G$ at $x$ respectively.
\par
The coefficients of a $s$-stage partitioned RK method are displayed in Table \ref{RK method}. For any $i \in \left \{ 1,...,s \right \} $, the coefficients $c_i= {\textstyle \sum_{j=1}^{s}}a_{ij}$ and $\bar c_i={\textstyle \sum_{j=1}^{s}}\bar a_{ij}$.
\begin{table}[]\label{RK method}\centering
	\caption{Coefficients of $s$-stage partitioned RK method}
	\renewcommand{\arraystretch}{1.5}
	\begin{tabular}{c|cccccc|ccc}
		$c_1$    & $a_{11}$ & $\cdots$ & $a_{1s}$ &  &  & $\bar c_1$ & $\bar a_{11}$ & $\cdots$ & $\bar a_{1s}$ \\
		$\vdots$ & $\vdots$ &          & $\vdots$ &  &  & $\vdots$   & $\vdots$      &          & $\vdots$      \\
		$c_s$    & $a_{s1}$ & $\cdots$ & $a_{ss}$ &  &  & $\bar c_s$ & $\bar a_{s1}$ & $\cdots$ & $\bar a_{ss}$ \\ \cline{1-4} \cline{7-10} 
		& $b_1$    & $\cdots$ & $b_s$    &  &  &            & $\bar{b}_1$    & $\cdots$ & $\bar{b}_s$   
	\end{tabular}
\end{table}

\section{The equivalence of direct approach and SPRK discretization for PMP canonical equations}\label{section2}	
As deduced in \cite{Hager2000RungeKuttaMI}, direct approach based on RK discretization is eventually equivalent with the corresponding SPRK integration for PMP canonical equations. In this section, we will reconstruct the equivalence by the analogue of continuous and discrete dynamic programming. The main purposes include two aspects: one is to reveal this equivalence in a more natural way; the other is to make some theoretical preparations for DLQR and ILQR methods introduced in Section \ref{section4} and \ref{section5}.
\par
Consider the unconstrained optimal control problem:\\
\textbf{Problem 2.1}
\begin{equation}\label{2problem}
	\min_{u}J= \min_{u}\int_{0}^{t_f} C(x(t),u(t))dt+\Phi (x(t_f))
\end{equation}
subject to
\begin{equation}\label{2state equation}	
	\frac{\mathrm{d} x}{\mathrm{d} t}=f(x,u),\ x(0)=a,\ t\in [0,t_f],
\end{equation}
where state variable $x(t)\in \mathbb{R}^n$, control $u(t)\in \mathbb{R}^m$, $f:\mathbb{R}^n\times \mathbb{R}^m\to \mathbb{R}^n$ and $J$ denotes the cost function with running cost $C:\mathbb{R}^n\times \mathbb{R}^m\to \mathbb{R}$ and terminal cost $\Phi:\mathbb{R}^n\to \mathbb{R}$. We assure suitable conditions are satisfied to ensure the existence, uniqueness and smoothness of the solution to this problem (More details can be seen in  \cite{Dontchev1995OptimalitySA,Hager2000RungeKuttaMI,Junge2005DiscreteMA}).
\par
To study this problem in a continuous form, we first introduce the value function
\begin{equation}\label{2 value function}
	V(t,x) \triangleq \min_{u_{[t,t_f]} } \left\{ \int_{t}^{t_f} C(x(s),u(s))ds+\Phi (x(t_f))     \right\}
\end{equation}
where $x(s)$ $(t\le s\le t_f)$, satisfying $x(t)=x$, is the state trajectory under control $u(s)$ $(t\le s\le t_f)$. Then according to the principle of optimality (refer to Section 5.1.2 in \cite{Liberzon2012CalculusOV}), we have
\begin{equation} \label{2 principle op}
V(t,x)=\min_{u_{[t,t+\Delta t]} }  \left \{ \int_{t}^{t+\Delta t} C(x(s),u(s)) ds +V(t+\Delta t,x(t+\Delta t)) \right \}      
\end{equation}
for every $(t,x)\in [0,t_f)\times \mathbb{R}^n$ and $\Delta t \in (0,t_f-t]$. We additionally assume the value function is smooth enough, then it can be determined by the Hamilton-Jacobi-Bellman (HJB) equation:
\begin{equation}\label{2 HJB}
	-V_t(t,x)=\min_{u}\left \{ C(x,u)+D_xV(t,x)\cdot f(x,u) \right \}  
\end{equation}
with boundary condition $V(t_f,x)=\Phi(x(t_f))$. As mentioned in Section 7.1.2 of \cite{Liberzon2012CalculusOV}, the characteristics of this partial differential equation are actually the canonical equations deduced from the maximum principle, namely,
\begin{subequations}
\begin{align}
	\frac{dx}{dt}&=f(x,u),\ x(0)=a, \label{2 cono1}\\
	\frac{dp}{dt}&=-D_x^Tf(x,u)p-D_xC(x,u) \label{2 cono2},\ p(t_f)= \Phi '( x(t_f) ),  \\
	D_u^T&f(x,u)p+D_uC(x,u)=0 \label{2 cono3}
\end{align}
\end{subequations}   
where the costate variable $p(t)$, $t\in [0,t_f]$, satisfies
\begin{equation}\label{2pk=DxV}
p(t)=D_x V(t,x).
\end{equation}
 Suppose the equation \eqref{2 cono3} implicitly determines a function $u=\hat{u}(x,p)$, then we can transform the canonical equations \eqref{2 cono1}-\eqref{2 cono3} into an autonomous form
\begin{subequations}
	\begin{align}
		\frac{dx}{dt}&=f(x,\hat{u}( x,p) ),\ x(0)=a, \label{2 cono4}\\
		\frac{dp}{dt}&=-D_x^Tf(x,\hat{u}( x,p))p-D_xC(x,\hat{u}( x,p)),\ p(t_f)= \Phi ' ( x(t_f) ) . \label{2 cono5}  
	\end{align}
\end{subequations} 
\par
To solve \textbf{Problem 2.1} numerically in the framework of direct approach, we construct the discrete time sequence $ \left \{ t_k=kh,\ k=0,\cdots ,t_f/h \right \} $ and utilize a $s$-stage RK method with coefficients $a_{ij},b_i$ for discretization, then we can get the following optimization problem:\\
\textbf{Problem 2.2} 
\begin{equation}\label{2dOCP}
	\min_{u_{ki}} J_{d}=\min_{u_{ki}} 
	\left\{ h{\textstyle \sum_{k=0}^{N-1}}  {\textstyle \sum_{i=1}^{s}}b_iC(x_{ki},u_{ki}) +\phi (x_N) \right\}	
\end{equation}
subject to 
\begin{subequations}
	\begin{align}
		x_{k+1}&=x_{k}+h {\textstyle \sum_{i=1}^{s}}b_{i}f(x_{ki},u_{ki}), &k&=0,...,N-1, \label{2dstate1} \\
		x_{ki}&=x_{k}+h {\textstyle \sum_{j=1}^{s}}a_{ij}f(x_{kj},u_{kj}), &k&=0,...,N-1,\ i=1,...,s. \label{2dstate2}
	\end{align}
\end{subequations}
Similar to the continuous case, we define the discrete value function for \textbf{Problem 2.2} as following
\begin{equation} \label{2d value function}
	V_k(x_k)=\min_{u_{li}}\left \{ h {\textstyle \sum_{l=k}^{N-1}} {\textstyle \sum_{i=1}^{s}}b_iC(x_{li},u_{li})+\Phi(x_N)   \right \}  ,\ 0\le k \le N-1,
\end{equation}
and $V_N(x_N)=\Phi(x_N).$ Then the discrete HJB equation with internal-stage controls is given by \cite{Ohsawa2011DiscreteHT}:
\begin{equation}\label{2d HJB}
	\begin{aligned}
	V_k(x_k) =& \min_{u_{k1},...,u_{ks}} \{	h{\textstyle \sum_{i=1}^{s}}b_iC(x_{ki},u_{ki})+V_{k+1}(x_{k+1})
	 \}\\
	=&\min_{u_{k1},...,u_{ks}} \left \{ h {\textstyle \sum_{i=1}^{s}}b_iC(x_{ki},u_{ki})+V_{k+1}(x_{k}+h {\textstyle \sum_{i=1}^{s}}b_if(x_{ki},u_{ki}))   \right \} , 
	\end{aligned}
\end{equation}
$0\le k \le N-1$, with boundary condition $V_N(x_N)=\Phi (x_N)$.
\par
With all the preparations above, now we can introduce the main deductive result of this section and the proof is left in Appendix \ref{App for section2}.
\begin{proposition}\label{2 key proposition}
If $b_i>0\ (1\le i \le s)$ and the time step $h$ is small enough, then the solution ${\left \{ u_{ki}:0\le k\le N-1,\ 1\le i\le s \right \} }$ of \textbf{Problem 2.2}, can be determined by the following discrete system:
\begin{subequations}
	\begin{align}
		x_{k+1}&=x_k+h {\textstyle \sum_{i=1}^{s}}b_i f(x_{ki},u_{ki}), \label{2 ad1} \\
		x_{ki}&=x_k+ h{\textstyle \sum_{j=1}^{s}}a_{ij}   f(x_{kj},u_{ki} ), \label{2 ad2} \\
		p_{k+1}&=p_k-h {\textstyle \sum_{i=1}^{s}}b_i(D_x^Tf(x_{ki},u_{ki})p_{ki}+D_xC(x_{ki},u_{ki}) ),\label{2 ad3} \\
		p_{ki}&=p_k-h {\textstyle \sum_{i=j}^{s}}\bar{a}_{ij}(D_x^Tf(x_{kj},u_{ki})p_{kj}+D_xC(x_{kj},u_{ki}) ),\label{2 ad4} 
	\end{align}
\end{subequations}
$0\le k\le N-1$, $1\le i\le s$, with boundary conditions $x_0=a$ and $p_N= \Phi'(x_N)$, where
\begin{subequations}
\begin{align}
	u_{ki}&=\hat{u}(x_{ki},p_{ki}), \label{2uki} \\
	p_k&= (V_k)'(x_k), \label{2 pk} \\
	\bar{a}_{ij}&=b_j-\frac{b_ja_{ji}}{b_i} \label{2 aij}.
\end{align}
\end{subequations} 

\end{proposition} 

\begin{remark}
The conclusion in Proposition \ref{2 key proposition} also can be obtained by introducing Lagrangian multiplies for state constraints \eqref{2dstate1} and \eqref{2dstate2} and then deducing the KKT conditions for \textbf{Problem 2.2}, as in \cite{Hager2000RungeKuttaMI}. In fact, according to \eqref{2 pk}, the Lagrangian multiplies is equivalent to the gradient of discrete value function, which can be taken as the discrete version of \eqref{2pk=DxV} in continuous setting. 
\end{remark} \par
Note that the expression \eqref{2 aij} is actually the symplectic condition for partitioned RK methods (see Theorem 4.6 in chapter VI of \cite{Hairer2004GeometricNI}), we can conclude that the solution of direct approach based on RK discretization also can be reached by the corresponding SPRK integration of canonical systems \eqref{2 cono4} and \eqref{2 cono5}. Consequently, it is natural to introduce the following definition which was first proposed in \cite{Hager2000RungeKuttaMI} :    
\begin{definition}
A RK method with coefficients $a_{ij},b_{i}>0$ is of order r in OCP, if the corresponding SPRK method with coefficients $a_{ij}$, $\bar{a}_{ij}=b_j-b_ja_{ji}/b_i$, $b_{i}=\bar{b}_i$ has order $r$.
\end{definition}

\par   
In Table \ref{2method A}-\ref{2method C}, we present three explicit RK methods (notated as A, B and C methods) and their respective adjoint part. One can verify that A, B and C methods are of order 2,3 and 4 in OCP, respectively (refer to chapter III.2 in \cite{Hairer2004GeometricNI} or Table 1 in \cite{Hager2000RungeKuttaMI}).      
	\begin{table}[H] \centering
		\caption{A Method (left) and its adjoint part (right)}
		\begin{tabular}{c|ccccc|cc}
			0 &     &     &  &  & 0 & 1/2 & -1/2 \\
			1 & 1   &     &  &  & 1 & 1/2 & 1/2  \\ \cline{1-3} \cline{6-8} 
			& 1/2 & 1/2 &  &  &   & 1/2 & 1/2 
		\end{tabular}
		\label{2method A}
	\end{table}
	\begin{table}[H] \centering
				\caption{B Method (left) and its adjoint part (right)}
	\begin{tabular}{c|cccccc|ccc}
		0   &     &     &     &  &  & 0   & 1/6 & -4/3 & 7/6  \\
		1/2 & 1/2 &     &     &  &  & 1/2 & 1/6 & 2/3  & -1/3 \\
		1   & -1  & 2   &     &  &  & 1   & 1/6 & 2/3  & 1/6  \\ \cline{1-4} \cline{7-10} 
		& 1/6 & 2/3 & 1/6 &  &  &     & 1/6 & 2/3  & 1/6 
	\end{tabular}
		\label{2method B}
	\end{table}

	\begin{table}[H] \centering
	\caption{C Method (left) and its adjoint part (right)}
		\begin{tabular}{c|ccccccc|cccc}
			0   &     &     &     &     &  &  & 0   & 1/6 & -2/3 & 1/3  & 1/6  \\
			1/2 & 1/2 &     &     &     &  &  & 1/2 & 1/6 & 1/3  & -1/6 & 1/6  \\
			1/2 & 0   & 1/2 &     &     &  &  & 1/2 & 1/6 & 1/3  & 1/3  & -1/3 \\
			1   & 0   & 0   & 1   &     &  &  & 1   & 1/6 & 1/3  & 1/3  & 1/6  \\ \cline{1-5} \cline{8-12} 
			& 1/6 & 1/3 & 1/3 & 1/6 &  &  &     & 1/6 & 1/3  & 1/3  & 1/6 
		\end{tabular}
		\label{2method C}
	\end{table}
\par
Though the equivalence revealed by Proposition \ref{2 key proposition} is indeed an elegant result, one restriction should be stated here. As is shown in Figure \ref{2 internal control}, the internal-stage control $u_{ki}$ is usually taken as the final result to approximate the optimal control at time $t_k+c_i h$. However, when the coefficients of the RK method satisfy $c_i=c_j, i\ne j$ (C method in Table \ref{2method C}, with coefficients $c_2=c_3=1/2$, for example), there may be two different internal-stage controls for the same time $t_k+c_i h$. Therefore, it seems to be more reasonable to add an \textit{internal-stage overlap constraint}  $u_{ki}=u_{kj}$. However, under this kind of constraint, the equivalence stated in Proposition \ref{2 key proposition} will not hold any longer and two difficulties will arise in practice: on the one hand, more stringent algebraic conditions must be satisfied by the RK method to ensure the convergence or high-order accuracy (refer to \cite{Dontchev2000SecondOrderRA}); on the other hand, the optimization problem after RK discretization will become more elaborate and the linear feedback technique proposed in Section \ref{section4} and \ref{section5} fails to work. Fortunately, analysis in the next section will show that we can achieve high-order accuracy without considering the \textit{internal-stage overlap constraint}.    
\begin{figure}[H]\label{2 internal control}
	\centering
	\includegraphics[width=1\linewidth]{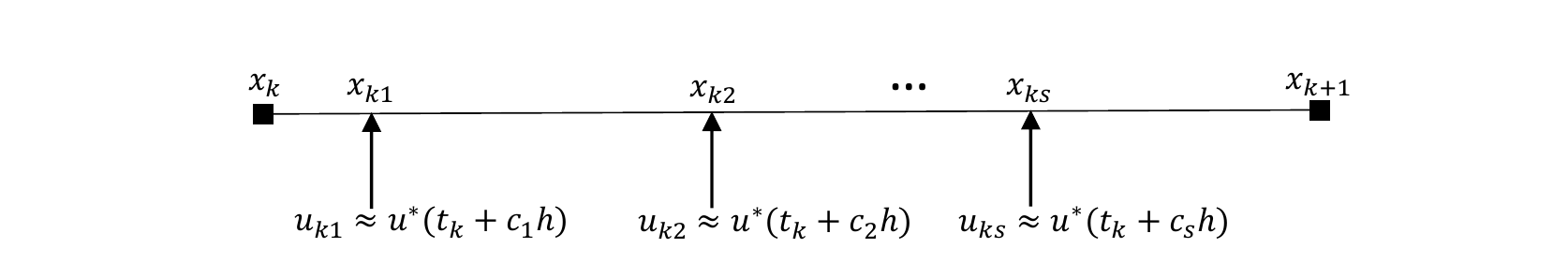}
	\caption{Internal-stage controls $u_{k1},...,u_{ks}$ in the time interval $[t_k,t_{k+1}]$.}
\end{figure}

\section{Error estimate and internal-stage control problem}\label{section3} Because of the equivalence stated in Section \ref{section2}, to analysis the convergence of a direct approach with RK discretization, it suffices to study the numerical integration for the canonical equations with corresponding SPRK method. By doing so, Hager proved the following theorem.      
\begin{theorem}[Error estimate for node values \cite{Hager2000RungeKuttaMI}] \label{3Hager}
	Suppose a $s$-stage RK method of order r in OCP is used in direct approach for solving problem \eqref{2problem}, then there exists unique solution $x_k,\ p_k$ for discrete system \eqref{2 ad1}-\eqref{2 ad4} when time step $h$ is small enough, such that
	\begin{equation}\label{3error1}
		\begin{aligned}
			\max_{0\le k\le N} \left || x_k-x^*(t_k) \right ||+ \left || p_k-p^*(t_k) \right ||+\left || \hat{u}(x_k,p_k)-u^*(t_k)\right ||
			\le O(h^r),
		\end{aligned}
	\end{equation}
	where $(x^*(t),p^*(t))$ is the exact solution of canonical equations \eqref{2 cono4} and \eqref{2 cono5}, and $u^*(t)=\hat{u}(x^*(t),p^*(t))$ is the optimal control for Problem \eqref{2problem}.
\end{theorem}
\par
The error estimate above is not about internal-stage controls $||u_{ki}-u^*(t_k+c_ih)||$ $(1\le i \le s)$ but about node control $||\hat{u}(x_k,p_k)-u^*(t_k)||$. As remarked and exemplified in \cite{Hager2000RungeKuttaMI}, the node control calculated by $u_k=\hat{u}(x_k,p_k)$ usually converges faster than the internal-stage controls  $``for\ Runge-Kutta\ schemes\ of\ third\ or\ fourth\ order"$. Indeed, as we will clarify later, such \textit{internal-stage control problem} is so common that some RK methods with second order cannot avoid it. Let us first give the following definition.
\begin{definition}\label{3definition}
	In the setting of Theorem \ref{3Hager}, the $i$-th $(i\in \left \{ 1,...,s \right \} )$ internal-stage control $u_{ki}$ has order $n$, if for all sufficiently regular Problems \eqref{2problem}, the error satisfies
	\begin{equation}\nonumber
		\max_{0\le k\le N-1}\left \{ ||u_{ki}-u^*(t_k+c_ih)|| \right \}    \le O(h^n).
	\end{equation}
\end{definition} 
As a main result of this section, we give the following order conditions for internal-stage controls and the proof is left to the end of this section.
\begin{theorem}[Order conditions for internal-stage controls]\label{3internal-stage control order conditions}
	For $i \in \left \{ 1,...,s \right \}  $, suppose the RK method in Theorem \ref{3Hager}, with order $r$ in OCP, satisfies
	\begin{subequations}
		\begin{align}
			{\textstyle \sum_{j=1}^{s}}a_{ij}c_j^{l-2}=\frac{c_i^{l-1}}{l-1},\ l=2,...,q_1\le r, \label{3 condition1}\\
			{\textstyle \sum_{j=1}^{s}}\bar a_{ij}c_j^{l-2}=\frac{c_i^{l-1}}{l-1},\ l=2,...,q_2\le r, \label{3 condition2}
		\end{align}
	\end{subequations}
	then the order of the $i$-th internal-stage control $u_{ki}$ is 
	\begin{equation}\label{3n condition number}
		n=\begin{cases}
			1  & \text{ if } c_i\ne \bar c_i, \\
			\min(q_1,q_2)  & \text{ if } c_i=\bar c_i.
		\end{cases}
	\end{equation}
Conversely, to ensure $u_{ki}$ has order $n$ $(2\le n \le r)$, the conditions $c_i=\bar{c}_i$ and $\min(q_1,q_2)=n $ must be satisfied for the RK method.	
\end{theorem}
\par
Some necessary explanations for Theorem \ref{3internal-stage control order conditions} are given in the remarks below.
\begin{remark}\label{3remark1}
The orders of internal-stage controls are less than or equal to the order of RK method used, that is why we emphasize $n \le r$ in Theorem \ref{3internal-stage control order conditions}. 	
\end{remark}

\begin{remark}\label{3remark2}
The condition $c_i=\bar c_i$ is crucial for the $i$-th internal-stage control to attain high order. Note that A, B and C methods in Table \ref{2method A}-\ref{2method C} satisfy $c_i=\bar c_i$ for all $i \in \left \{ 1,...,s \right \}$, consequently all the internal-stage controls produced by them have at least second order. Nevertheless, for the second-order implicit trapezoidal method, as is shown in Table \ref{3 implicit trapezoidal} that $c_1=0 \ne \bar c_1=1/2$ and $c_2=1 \ne \bar c_2=1/2$, the two internal-stage controls $u_{k1}$ and $u_{k2}$ have just order 1. 
\end{remark}

\begin{remark}\label{3remark3}
The internal-stage controls may converge faster than the result in Theorem \ref{3internal-stage control order conditions} when Problem \eqref{2problem} is in some special form. For example, when the function $u=\hat{u}(x,p)$ determined by the third of canonical equation \eqref{2 cono3} is independent with the state variable $x$, namely $u=\hat{u}(p)$, then the accuracy order of $i$-th internal-stage control will be $q_2$ instead of $\min(q_1,q_2)$. The reason can be explained easily after we give the proof of Theorem \ref{3internal-stage control order conditions}. However, as we have stated in Definition \ref{3definition}, we discuss the orders of internal-stage controls in a general problem setting rather than some special cases.          
\end{remark}

\begin{table}[t]\centering \label{3 implicit trapezoidal}
		\caption{Implicit trapezoidal method $(left)$ and its adjoint part $(right)$.}
	\begin{tabular}{c|ccccc|cc}
		0 &     &     &  &  & 1/2 & 1/2 & 0   \\
		1 & 1/2 & 1/2 &  &  & 1/2 & 1/2 & 0   \\ \cline{1-3} \cline{6-8} 
		& 1/2 & 1/2 &  &  &     & 1/2 & 1/2
	\end{tabular}

\end{table}

To show the validity of Theorem \ref{3internal-stage control order conditions}, we give a simple example here.

 \begin{table}[t] \centering
 	\caption{ $q_1$ and $q_2$ of method A, B and C}
 	\label{3 table q1q2}
 	\begin{tabular}{c|cc|ccc|cccc}
 		& \multicolumn{2}{c|}{A} & \multicolumn{3}{c|}{B} & \multicolumn{4}{c}{C} \\ \hline
 		$i$   & 1          & 2         & 1      & 2     & 3     & 1   & 2   & 3   & 4   \\
 		$q_1$ & 2          & 2         & 3      & 2     & 2     & 4   & 2   & 2   & 3   \\
 		$q_2$ & 2          & 2         & 2      &2     & 3     & 3   & 2   & 2   & 4  
 	\end{tabular}
 \end{table}

\begin{table}[t] \centering 	\label{3 Max error2}
		\caption{Max error of internal controls $(\max_{k=0,...,N-1}||u^*(kh+c_ih)-u_{ki}||,\ i=1,...,s ) $ in Example \ref{3 example1}.}
	\resizebox{\textwidth}{12mm}{	\begin{tabular}{cccccccccc}
			\hline
			$h$  & $A:u_{k1}$ & $A:u_{k2}$ & $B:u_{k1}$ & $B:u_{k2}$ & $B:u_{k3}$ & $C:u_{k1}$ & $C:u_{k2}$ & $C:u_{k3}$ & $C:u_{k4}$ \\ \hline
			0.1  & 2.40E-03   & 2.48E-03   & 1.28E-03   & 9.23E-04   & 2.61E-03   & 1.40E-04   & 2.67E-04   & 3.02E-04   & 1.11E-05   \\
			0.05 & 5.94E-04   & 6.56E-04   & 4.20E-04   & 2.60E-04   & 6.42E-04   & 1.76E-05   & 7.02E-05   & 7.45E-05   & 1.55E-06   \\
			0.04 & 3.79E-04   & 4.25E-04   & 2.82E-04   & 1.70E-04   & 4.09E-04   & 9.05E-06   & 4.53E-05   & 4.75E-05   & 8.08E-07   \\
			0.02 & 9.43E-05   & 1.09E-04   & 7.67E-05   & 4.41E-05   & 1.01E-04   & 1.13E-06   & 1.15E-05   & 1.18E-05   & 1.05E-07   \\
			0.01 & 2.35E-05   & 2.74E-05   & 1.99E-05   & 1.12E-05   & 2.52E-05   & 1.42E-07   & 2.91E-06   & 2.94E-06   & 1.34E-08   \\ \hline
	\end{tabular}}
\end{table}
\begin{example}\label{3 example1}
	Consider the problem
	\begin{equation}\nonumber
		\min_{u} \int_{0}^{1} \frac{1}{2}x^2+\frac{1}{2}xu+\frac{1}{2}u^2dt,
	\end{equation}
	subject to
	\begin{equation}\nonumber
		\dot{x}=u,\ x(0)=1.
	\end{equation}
	The optimal control of this problem is given by
	\begin{equation}\nonumber
		u^*(t)=\frac{0.5e^t-1.5e^{2-t}}{0.5+1.5e^2}.
	\end{equation}
The max errors of internal-stage controls generated by A,B and C methods $($Table \ref{2method A}-\ref{2method C}$)$ are shown in Table \ref{3 Max error2}. We just discuss the result of C method here. According to Theorem \ref{3internal-stage control order conditions} and Table \ref{3 table q1q2}, the orders of four internal-stage controls  $u_{k1},\ u_{k2},\ u_{k3},\ u_{k4}$ should have the same numbers as $\min (q_1,q_2): 3,\ 2,\ 2,\ 3$, which coincides with the result in Table \ref{3 Max error2}.
\end{example}
\par
High-order explicit RK methods are advantageous in practice for saving computational cost. Thus it is necessary to study whether there exists $s$-stage explicit RK method with $s$-order in OCP can avoid \textit{internal-stage control problem} $(s=2,3,4)$. The reason why we restrict $s\le4 $ is that according to the well-known result of Butcher (see Theorem 5.1 in \cite{hairer1993solving}), explicit RK method of $s$-stage cannot attain order $s$ when $s>4$. To save computational payment, we do not concern the explicit RK method with stage number greater than the order. 
 \par 
For the second order explicit RK method with coefficients $a_{21}, b_1$ and $b_2$, to ensure the first internal-stage control $u_{k1}$ attain order 2, the following conditions must be satisfied:
\begin{equation}\nonumber
{\textstyle \sum_{i=1}^{2}}b_{i}=1,\  {\textstyle \sum_{i=1}^{2}}b_ic_i=\frac{1}{2},\ c_1=\bar{c}_1 . 	
\end{equation}
The first two conditions can ensure the RK method attain second order in OCP (see Table 1 in \cite{Hager2000RungeKuttaMI}) and the third condition can ensure $u_{k1}$ achieve second order. After some simple calculations, we can check that only A method (Table \ref{2method A}) with coefficients $a_{21}=1, b_1=b_2=\frac{1}{2}$ satisfies these conditions. In addition, since $c_2=\bar c_2$ is also satisfied for A method, we can conclude that the second internal-stage control $u_{k2}$ also has order 2.
\par
  
  \begin{table}[t]\centering \nonumber
  		\caption{3-stage explicit RK method with order 3 in OCP}
  	\label{3 explicit3}
  	\begin{tabular}{c|ccc}
  		0     &                                     &                             &                           \\
  		$c_2$ & $c_2$                               &                             &                           \\
  		1     & $\frac{3c_2-1-3c_2^2}{c_2(2-3c_2)}$ & $\frac{1-c_2}{c_2(2-3c_2)}$ &                           \\ \hline
  		& $\frac{c_2-1/3}{2c_2}$              & $\frac{1}{6c_2(1-c_2)}$     & $\frac{2-3c_2}{6(1-c_2)}$
  	\end{tabular}
  \end{table}

For the third or fourth order explicit methods, we first verify the conditions $\bar{c}_i=c_i,\ i=1,...,s$ are satisfied. 
\begin{proposition}\label{3 cc}
	For s-stage explicit RK methods with s order in OCP $(s=3,4)$, the conditions $c_i=\bar{c}_i,\ (i=1,...,s)$ hold.
\end{proposition}
\begin{proof}
	Since 
	\begin{equation}\nonumber
		\bar{c}_i= {\textstyle \sum_{j=1}^{s}}\bar{a}_{ij}= {\textstyle \sum_{j=1}^{s}}(b_j-\frac{b_ja_{ji}}{b_i}  ),\ i=1,...,s,
	\end{equation}
	to prove $\bar{c_i}=c_i$, it suffices to verify:
	\begin{equation}\label{3 c=c}
		b_i {\textstyle \sum_{j=1}^{s}} b_j- {\textstyle \sum_{j=1}^{s}}b_ja_{ji}=b_i - {\textstyle \sum_{j=1}^{s}}b_ja_{ji}=b_ic_i,\ i=1,...,s,
	\end{equation}
	where we have applied the first order condition for RK method: $ {\textstyle \sum_{j=1}^{s}}b_j=1 $. 
	\par
To ensure a 3-stage explicit RK attain third order in OCP, the unique form is shown in Table \ref{3 explicit3} (refer \cite{Hager2000RungeKuttaMI}). By a direct computation, we can verify that the coefficients in Table \ref{3 explicit3} satisfy \eqref{3 c=c}. For the fourth order explicit RK methods, it is presented in \cite{butcher1987numerical} (P178) that \eqref{3 c=c} is automatically satisfied.	
\end{proof}
\par
As an immediate result of Remark \ref{3remark2} and Proposition \ref{3 cc}, all internal-stage controls generated by explicit RK methods with third or fourth order in OCP have at least second order. 
\par 
For a third order explicit method, to ensure the $i$-th internal-stage control $u_{ki}$ $(i \in \{1,...,s\})$ owns third order, two additional conditions ($l=3$ in \eqref{3 condition1} and \eqref{3 condition2}) must be satisfied. However, there is only one independent coefficient $c_2$, as shown in Table \ref{3 explicit3}, thus it is impossible to avoid order reduction. A similar reasoning can be used for the fourth order explicit RK methods. Thus we can conclude that the s-order explicit RK methods with s stages cannot avoid the \textit{internal-stage control problem}. 
\par  
When we choose the internal-stage control as final result, the order reduction problem restricts the choice of RK methods. Of course, one can construct high-order implicit RK methods satisfying enough conditions in Theorem \ref{3internal-stage control order conditions}. But this issue is not concerned in this paper. Another way to solve the \textit{internal-stage control problem}, as suggested in the introduction of \cite{Hager2000RungeKuttaMI}, is taking the node control $u_k=\hat{u}(x_k,p_k)$ as final result instead of the internal-stage controls. According to the convergence analysis in Theorem \ref{3Hager}, $u_k$ have the same accuracy order as the RK method itself. This handling is suitable for more general RK discretization, and the additional computational cost, as we will show in the later sections concerning specific numerical algorithms, will not be huge.      
\par
We conclude this section with the proof of theorem \ref{3internal-stage control order conditions}. To begin with, we give two lemmas on error estimate of internal-stage values generated by RK method.
\par
\begin{lemma}\label{3lemma1}
Consider a differential equation $\dot{y}=f(y),\ y(t_0)=y_0$ with $f:\ \mathbb{R}^n\to \mathbb{R}^n$ smooth enough. Apply a $s$-stage Runge-Kutta method with step size $h$ to this equation, then for any $i \in \{1,...,s\}$, the internal-stage value $y_i$, calculated by
	\begin{equation}\nonumber
		y_i=y_0+ h{\textstyle \sum_{j=1}^{s}}a_{ij}f(y_j),
	\end{equation} 
	satisfies $||y(t_0+c_ih)-y_i||=O(h^{n})$ if and only if
	\begin{equation}\label{3condition ode1}
		{\textstyle \sum_{j=1}^{s}}a_{ij}c_j^{l-2}=\frac{c_i^{l-1}}{l-1},\ 	for\ l=2,...,n.
	\end{equation}
\end{lemma}\label{lemma3.3}

\begin{proof}
The details of root tree theory in this proof can be seen in Chapter III of \cite{Hairer2004GeometricNI}, we just list some formulas when necessary. Regard the internal-stage value $y_i(h)= y_0+h{\textstyle \sum_{j=1}^{s}}f(y_j)$ as a function with variable $h$. We can express the $l$-th derivatives of $y(t)$ and $y_i(h)$ as following
	\begin{subequations}
		\begin{align}
		y^{(l)}(t_0)=& {\textstyle \sum_{|\tau |=l}^{ }}\alpha (\tau )F(\tau)(y_0), \nonumber \\
			y_i^{(l)}(0)=& {\textstyle \sum_{|\tau |=l}^{}}\gamma (\tau )u_i(\tau )\alpha (\tau )F(\tau )(y_0), \nonumber	
		\end{align}
 	\end{subequations}
where $F(\tau)$ is the elementary differential about tree $\tau$, $\alpha(\tau)$ and $\gamma(\tau)$ are the integer coefficients, and $u_i(\tau)$ is the factor about $a_{ij}$. According to the Taylor expansion, $||y(t_0+c_ih)-y_i||=O(h^{n})$ if and only if $\gamma(\tau)u_i(\tau)=c_i^{l}$ for each $\tau$ with tree order $1 \le|\tau|=l\le n-1$. Specially, when we choose the tree $\tau=[\tau _1,...,\tau _{l-1}]$ where $\tau_1=...=\tau_{l-1}=\bullet$, according to the formulas (1.14) and (1.15) in 	Chapter III of \cite{Hairer2004GeometricNI}:
\begin{subequations}
	\begin{align}
		\gamma (\tau )&=|\tau|\gamma (\tau _1)...\gamma (\tau _{l-1}), \nonumber \\
		u_i(\tau )&= {\textstyle \sum_{j=1}^{s}}a_{ij}u_j(\tau_1 )...u_j(\tau _{l-1}),
		 \nonumber
	\end{align}
\end{subequations}
and $\gamma(\bullet)=1,\ u_j(\bullet)= {\textstyle \sum_{k=1}^{s}}a_{jk}=c_j$, we have
\begin{subequations}
	\begin{align}
	\gamma (\tau )&=l, \nonumber \\
	u_i(\tau )&= {\textstyle \sum_{j=1}^{s}}a_{ij} c_j^{l-1},\nonumber
	\end{align}
\end{subequations}
thus the condition 
\begin{equation}\nonumber
	l \cdot {\textstyle \sum_{j=1}^{s}}a_{ij}c_j^{l-1}=c_i^l 
\end{equation}
must be satisfied for $ 1\le l\le n-1$, which proves the necessity. 
\par
Sufficiency can be proved by induction. For $|\tau|=1$, $\tau $ can only be $\bullet$, thus the equality $\gamma(\tau)u_i(\tau)=c_i$ is just the condition \eqref{3condition ode1} with $l=2$. Suppose $\gamma(\tau)u_i(\tau)=c_i^{|\tau|}$ is satisfied for $|\tau|=1,2,...,N<n-1$. For $\tau=[\tau_1,...,\tau_m]$ with $|\tau|=1+|\tau_1|+...+|\tau_m|=N+1$, we have $|\tau_w|\le N$ for $w=1,...,m$. Then
	\begin{equation}\label{3.4}
		\begin{aligned}
			\gamma (\tau )u_i(\tau )&=|\tau | {\textstyle \sum_{j=1}^{s}}a_{ij} {\textstyle \prod_{w= 1}^{m}} \gamma (\tau_w)u_j(\tau _w)\\
			& =|\tau| {\textstyle \sum_{j=1}^{s}}a_{ij}c_j^{|\tau _1|+...+|\tau _m|}\\
			& =(N+1) {\textstyle \sum_{j=1}^{s}}a_{ij}c_j^N=c_i^{N+1},
		\end{aligned}
	\end{equation} 
where the last equality is due to the condition \eqref{3condition ode1} with $l=N+2$. Now we have proven the sufficient part of the lemma.
\end{proof}
The analogous result for partitioned Runge-Kutta methods requires only minor modifications. The proof is quite similar to Lemma \ref{3lemma1} and so is omitted.
\begin{lemma}\label{lemma3.4}
Consider differential equations 
	\begin{equation}\nonumber
		\begin{aligned}
			\dot{x}&=f(x,y),\ x(t_0)=x_0,\\
			\dot{y}&=g(x,y),\ y(t_0)=y_0
		\end{aligned}
	\end{equation}
where $f$ and $g:\mathbb{R}^n \times \mathbb{R}^n\to \mathbb{R}^n$ are smooth enough. When a partitioned RK method with coefficients $a_{ij},\ b_i,\ \bar{a}_{ij},\ \bar{b}_i$ satisfying $c_i= {\textstyle \sum_{j=1}^{s}}a_{ij}=\bar c_i= {\textstyle \sum_{j=1}^{s}}\bar a_{ij} $ is applied to numerical integration, the intermediate values 
\begin{subequations}
	\begin{align}
	x_{i}&=x_0+h {\textstyle \sum_{j=1}^{s}}a_{ij}f(x_{j},y_{j}), \nonumber \\
y_{i}&=y_0+h {\textstyle \sum_{j=1}^{s}}\bar{a}_{ij}g(x_{j},y_{j}) \nonumber
	\end{align}
\end{subequations}
satisfy $||x(t_0+c_ih)-x_{i}||\le O(h^{q_1} )$ and $||y(t_0+c_ih)-y_{i}||\le O(h^{q_2})$ if and only if 
\begin{subequations}
	\begin{align}
		{\textstyle \sum_{j=1}^{s}}a_{ij}c_j^{l-2}=\frac{c_i^{l-1}}{l-1},\ l=2,...,q_1, \nonumber\\
		{\textstyle \sum_{j=1}^{s}}\bar a_{ij}c_j^{l-2}=\frac{c_i^{l-1}}{l-1},\ l=2,...,q_2. \nonumber
	\end{align}
\end{subequations}	
\end{lemma} 

\begin{remark} \label{3 remark}
When $c_i \ne \bar{c}_i$ in Lemma \ref{3.4}, we can express the intermediate value $y_i$ by Taylor expansion
\begin{equation}\nonumber
\begin{aligned}
	y_i=&y_0+h {\textstyle \sum_{i=1}^{s}}\bar a_{ij}g(x_0,y_0)+O(h^2)\\
	=&y_0+h\bar c_ig(x_0,y_0)+O(h^2),
\end{aligned}
\end{equation}
then it is an immediate result that
\begin{equation}\nonumber
||y(t_o+c_ih)-y_i||\le h||(c_i-\bar c_i)g(x_0,y_0)||+O(h^2)=O(h).
\end{equation}
\end{remark}
Now we are in a position to prove Theorem \ref{3internal-stage control order conditions}.  
\begin{proof}
For simple notation, we denote the vector field determined by canonical equations \eqref{2 cono4} and \eqref{2 cono5} as $F(z)$ where $z(t)=(x(t),p(t))$ and the exact flow with any initial condition $z(t_0)=z_0$ as $\phi (t;t_0,z_0)$. We also denote $L$ and $L_0$ as the Lipschitz constants for $F(z)$ and $\hat{u}(z)$ in a neighbourhood of the trajectory
$\left \{ z^*(t)=(x^*(t),p^*(t))\in \mathbb{R}^{2n}: 0\le t \le t_f \right \} $ respectively. Suppose the conditions in Theorem \ref{3internal-stage control order conditions} are satisfied, then the difference between $z_{ki}=(x_{ki},p_{ki})$ and $z^*(t_k+c_ih)$ satisfies 
 \begin{equation}\label{3final1}
 	\begin{aligned}
 		||z_{ki}-z^*(t_k+c_ih)||\le &||z_{ki}-\phi(t_k+c_ih;t_k,z_k)||\\
 		&+||\phi(t_k+c_ih;t_k,z_k)-\phi(t_k+c_ih;t_k,z^*(t_k)|| \\
 		\le &O(h^n)+ ||\phi(t_k+c_ih;t_k,z_k)-\phi(t_k+c_ih;t_k,z^*(t_k)||
 \end{aligned} \end{equation}
where $n$ satisfies \eqref{3n condition number} according to the Lemma \ref{lemma3.4} and Remark \ref{3 remark}. Since 
\begin{equation}\nonumber
	\begin{aligned}
		||\phi (t;t_k,z_k)-\phi(t;t_k,&z^*(t_k))|| \\
		=&||z_k-z^*(t_k)+\int_{t_k}^{t}F(\phi (\tau;t_k,z_k) ) -F(\phi(\tau;t_k,z^*(t_k)) )d\tau||\\
		\le&||z_k-z^*(t_k)||+L\int_{t_k}^{t}|| \phi (\tau;t_k,z_k)-\phi(\tau;t_k,z^*(t_k))||d\tau
	\end{aligned}
\end{equation}
for $t_k\le t \le t_k+c_ih$, according to Gronwall's inequality of integral form, we can obtain
	\begin{equation}\label{3final2}
		\begin{aligned}
	||\phi (t_k+c_ih;t_k,z_k)-\phi(t_k+c_ih;t_k,z^*(t_k))|| \le& (1+c_ihLe^{c_ihL})||z_k-z^*(t_k)|| \\
	\le& (1+c_ihLe^{c_ihL}) O(h^r)
		\end{aligned}
	\end{equation}
where $||z_k-z^*(t_k)|| \le O(h^r)$ is due to the Theorem \ref{3Hager}. Plugging \eqref{3final2} into \eqref{3final1}, an error estimate for $z_{ki}$ is given by  
	\begin{equation}\nonumber
		\begin{aligned}
			||z_{ki}-z^*(t_k+c_ih)||
			\le (1+c_ihLe^{c_ihL})O(h^r)+O(h^{n}).
		\end{aligned}
	\end{equation}
Finally, the convergence rate of $u_{ki}$ is obtained by
 	\begin{equation}\nonumber
		||u_{ki}-u^*(t_k+c_ih)||=||\hat{u}(z_{ki})-\hat{u}(z^*(t_k+c_ih))||\le L_0||z_{ki}-z^*(t_k+c_ih)||.
	\end{equation}   
\end{proof}	
Recall that, in Remark \ref{3remark3}, we have referred a special case $u=\hat{u}(p)$ where the accuracy order of internal-stage control is determined by $q_2$ instead of $\min(q_1,q_2)$. This can be verified easily with variable $z=(x,p)$ substituted by $p$ in the proof above and note that the internal-stage costate has order $n=q_2$ according to Lemma \ref{lemma3.4}.

\section{DLQR method based on high-order RK discretization for linear quadratic problems}\label{section4} 
In this  section, we restrict the \textbf{Problem 2.1} as the following linear quadratic form
\begin{equation}\label{4objective}\nonumber
	\min_{u}J= \min_{u}\int_{0}^{t_f}( \frac{1}{2}x^{T}Qx+\frac{1}{2}u^{T}Ru)dt+ \frac{1}{2}(x(t_f))^{T}Mx(t_f)
\end{equation}
subject to
\begin{equation}\label{4state}	\nonumber
	\dot{x}=Ax+Bu,\ x(0)=a,\ t\in [0,t_f],
\end{equation}
where $Q\ge0,R>0$ and $M\ge0$ are symmetric matrices.  \par
Based on the discussion of Section \ref{section3}, we will calculate the node control to avoid the order reduction problem of internal-stage control. To achieve this goal, it is necessary to deduce the function $u=\hat{u}(x,p)$. In the present problem setting, the third canonical equation \eqref{2 cono3} is given by
\begin{equation} \nonumber
	B^Tp+Ru=0,
\end{equation} 
and thus we can obtain the linear feedback law:
\begin{equation}\label{4 optimal control law}
	\hat u(x,p)=-R^{-1}B^Tp.
\end{equation}

\subsection{Algorithm: DLQR method based on high-order RK discretization}\label{section 4.1}
Our method is divided into four steps:   
\par
\begin{enumerate}[(1)]
	\item
\textbf{Discretization and simplification}. As the starting step, we apply a $s$-stage RK method with coefficients $a_{ij},\ b_i$ to discretization, and then obtain the following optimization problem with linear constraints:
\begin{equation} \label{4d objective}
	\begin{aligned}
\min_{u_{ki}} J_d=\min_{u_{ki}}  \frac{h}{2} {\textstyle \sum_{k=0}^{N-1}}{\textstyle \sum_{i=1}^{s}}b_i (x_{ki}^TQx_{ki}+u_{ki}^TQu_{ki})+\frac{1}{2}x_N^TMx_N 
	\end{aligned}
\end{equation}
subject to
	\begin{subequations}
\begin{align}
	x_{k+1}&=x_k+h {\textstyle \sum_{i=1}^{s}}b_{i} (Ax_{ki}+Bu_{ki}), &0&\le k \le N-1, \label{4d state1} \\
	x_{ki}&=x_k+h {\textstyle \sum_{j=1}^{s}}a_{ij} (Ax_{kj}+Bu_{kj}), &0&\le k \le N-1,\ 1\le i \le s.
	\label{4d state2}   
\end{align}
	\end{subequations}
For convenience, let us introduce the notations $X_k=(x_{k1},...,x_{ks})$ and $U_k=(u_{k1},...,u_{ks})$ for the vectors composed by all the internal-stage state and control values during $[t_k,t_{k+1}]$ respectively, then we can transform the discrete cost function \eqref{4d objective} into a concise form:
\begin{equation} \label{4d obj}
J_d=\frac{1}{2} {\textstyle \sum_{k=0}^{N-1}} ( X_k^TQ_hX_k+U_k^TR_hU_k)+\frac{1}{2}x_N^TMx_N    , 
\end{equation}
where 
\begin{equation}\label{Q_h R_h}
Q_h=h\begin{bmatrix}
	b_1Q &  & 0\\
	& \ddots & \\
	0 &  &b_sQ
\end{bmatrix},\ 
R_h=h\begin{bmatrix}
	b_1R &  & 0\\
	& \ddots & \\
	0 &  &b_sR
\end{bmatrix}.
\end{equation}
Further, by introducing the notations
\begin{equation}\nonumber
Z=\begin{bmatrix}
	I_n(1) \\
	\vdots  \\
	I_n(s)
\end{bmatrix},\ 
\mathcal{A}=h\begin{bmatrix}
	a_{11}A  & \cdots  &a_{1s}A \\
	\vdots   &  & \vdots \\
	a_{s1}A & \cdots & a_{ss}A
\end{bmatrix} ,\ 
\mathcal{B}=h\begin{bmatrix}
	a_{11}B  & \cdots  &a_{1s}B \\
	\vdots   &  & \vdots \\
	a_{s1}B & \cdots & a_{ss}B
\end{bmatrix},
\end{equation}
we can rewrite the constraints \eqref{4d state2} as follows
\begin{equation}\nonumber
	X_k=Z x_k +\mathcal{A}X_k+\mathcal{B}U_{k},\ 0\le k \le N-1,
\end{equation} 
or equivalently, denoting $E=(I_{sn}-\mathcal{A})^{-1}Z$ and $F=(I_{sn}-\mathcal{A})^{-1} \mathcal{B}$, we have
\begin{equation}\label{4d Xk}
	X_k=Ex_k+FU_k,\ 0\le k\le N-1.
\end{equation}
Note that the inverse matrix of $I_{sn}-\mathcal{A}$ exists when time step $h$ is small enough. Substituting \eqref{4d Xk} into \eqref{4d state1}, we can obtain the following formula after combing like terms
\begin{equation}\label{4d xk+1}
	x_{k+1}=G x_k +H U_k,\ 0\le k \le N-1,
\end{equation}
where
\begin{subequations}
	\begin{align}
		G&=I_n+h\begin{bmatrix}
			b_1A  & \cdots &b_sA
		\end{bmatrix}E, \nonumber \\
		H&=h\begin{bmatrix}
			b_1A  & \cdots &b_sA
		\end{bmatrix}F+h\begin{bmatrix}
			b_1B  & \cdots &b_sB
		\end{bmatrix}. \nonumber
	\end{align}
\end{subequations}
\item
\textbf{Construction of feedback law.} Denoting the discrete value function as \\
$V_k(x_k)\ (0\le k \le N)$, then the discrete HJB equation is given by
\begin{equation}\label{4d HJB}
	\begin{aligned}
		V_k(x_k)=&  \min_{U_k} \{\frac{1}{2}X_k^TQ_hX_k+\frac{1}{2}U_k^TR_hU_k+V_{k+1}(x_{k+1})  \} \\
		=&\min_{U_k} \{\frac{1}{2}(Ex_k+FU_k)^TQ_h  (Ex_k+FU_k)+\frac{1}{2}U_k^TR_hU_k \\
		&+V_{k+1}(Gx_k+HU_k) \},\ 0\le k \le N-1,
	\end{aligned}
\end{equation}
and the boundary condition is $V_N(x_N)=\frac{1}{2}x_N^T M x_N$. Suppose the value function has a quadratic form
\begin{equation}\label{4Vk}
	V_k(x_k)=\frac{1}{2}x_k^T M_k x_k,\ 0\le k \le N,
\end{equation}
where $M_k$ is symmetric, then according to the boundary condition, we have $M_N=M$. Plug \eqref{4Vk} into the discrete HJB equation \eqref{4d HJB} and let the derivative of the right hand of \eqref{4d HJB} with respect to $U_k$ vanish, we can obtain
\begin{equation}\nonumber
	(F^TQ_hF+R_h+H^TM_{k+1}H)U_k+(F^TQ_hE+H^TM_{k+1}G)x_k=0, 
\end{equation}
$0\le k \le N-1$, it follows that $U_{k}$ should satisfies a linear feedback law
\begin{equation}\label{4d linear feedback law}
	U_k=L_k x_k,\ 0\le k\le N-1,
\end{equation}
where
\begin{equation} \nonumber
L_k=-(F^TQ_hF+R_h+H^TM_{k+1}H)^{-1}(F^TQ_hE+H^TM_{k+1}G).
\end{equation}
\item
\textbf{Calculation of the value function.} Plugging \eqref{4d linear feedback law} into the discrete HJB equation \eqref{4d HJB} and comparing both sides of the equality, we obtain the iteration formula for matrix $M_k$:
\begin{equation}\nonumber
M_k=(E+FL_k)^TQ_h(E+FL_k)+L_k^TR_hL_k+(G+HL_k)^TM_{k+1}(G+HL_k),
\end{equation}
$k=N-1,...,0$. With $M_N=M$ known, $M_{N-1},...,M_{0}$ can be calculated recursively. 
\item
\textbf{Calculation of the node control.} Substituting \eqref{4d linear feedback law} into \eqref{4d xk+1}, we can calculate the discrete state trajectory by
\begin{equation}\nonumber
	x_{k+1}= (G+ H L_k)x_k,\ x_0=a,\ 0\le k \le N-1.
\end{equation}  
With $M_k$ and $x_k$ known, we can obtain the discrete costates by \eqref{2 pk}:
\begin{equation}\nonumber
	p_k=(V_k)'(x_k)=M_k x_k,\ 0\le k \le N.
\end{equation}
Finally, the node controls $u_k\ (0\le k \le N)$ can be calculated by plugging $p_k$ into the formula \eqref{4 optimal control law}, namely,
 \begin{equation}\nonumber
 	u_k=-R^{-1}B^Tp_k,\ 0\le k \le N.
 \end{equation}
\end{enumerate}

\subsection{Numerical example: Linear spring oscillator}
Consider a linear oscillator under control with following state equations and initial conditions: 
\begin{equation}\nonumber
	\begin{bmatrix}
		\dot{x}_1 \\
		\dot{x}_2 
	\end{bmatrix}=
	\begin{bmatrix}
		0 &1 \\
		-1  &0
	\end{bmatrix}
	\begin{bmatrix}
		x_1\\x_2
		
	\end{bmatrix}
	+\begin{bmatrix}
		1\\0
		
	\end{bmatrix}u,\ 
\begin{bmatrix}
	x_1(0)\\
	x_2(0)
\end{bmatrix}=\begin{bmatrix}
	1\\
	1
\end{bmatrix}.
\end{equation}
 Now search for control $u(t),\ t\in [0,40],$ to minimize the cost function:
\begin{equation}\nonumber
	J=\int_{0}^{40}(\frac{1}{2} x^Tx+1.5u^2)dt +5x(1)^Tx(1), 
\end{equation}
where $x=(x_1,x_2)$.
\par
We apply explicit Euler, implicit trapezoidal and B methods in this numerical example. As Figure \ref{LQR_error} shows, DLQR method based on high order RK discretization shows remarkable advantage in accuracy compared to the low-order case. In particular, the implicit trapezoidal method, as we have demonstrated in Remark \ref{3remark2}, which produces the internal-stage controls with only first order, generates the node control with second order accuracy in this example.
\begin{figure}[H]\label{LQR_error}
	\centering
	\includegraphics[width=0.6\linewidth]{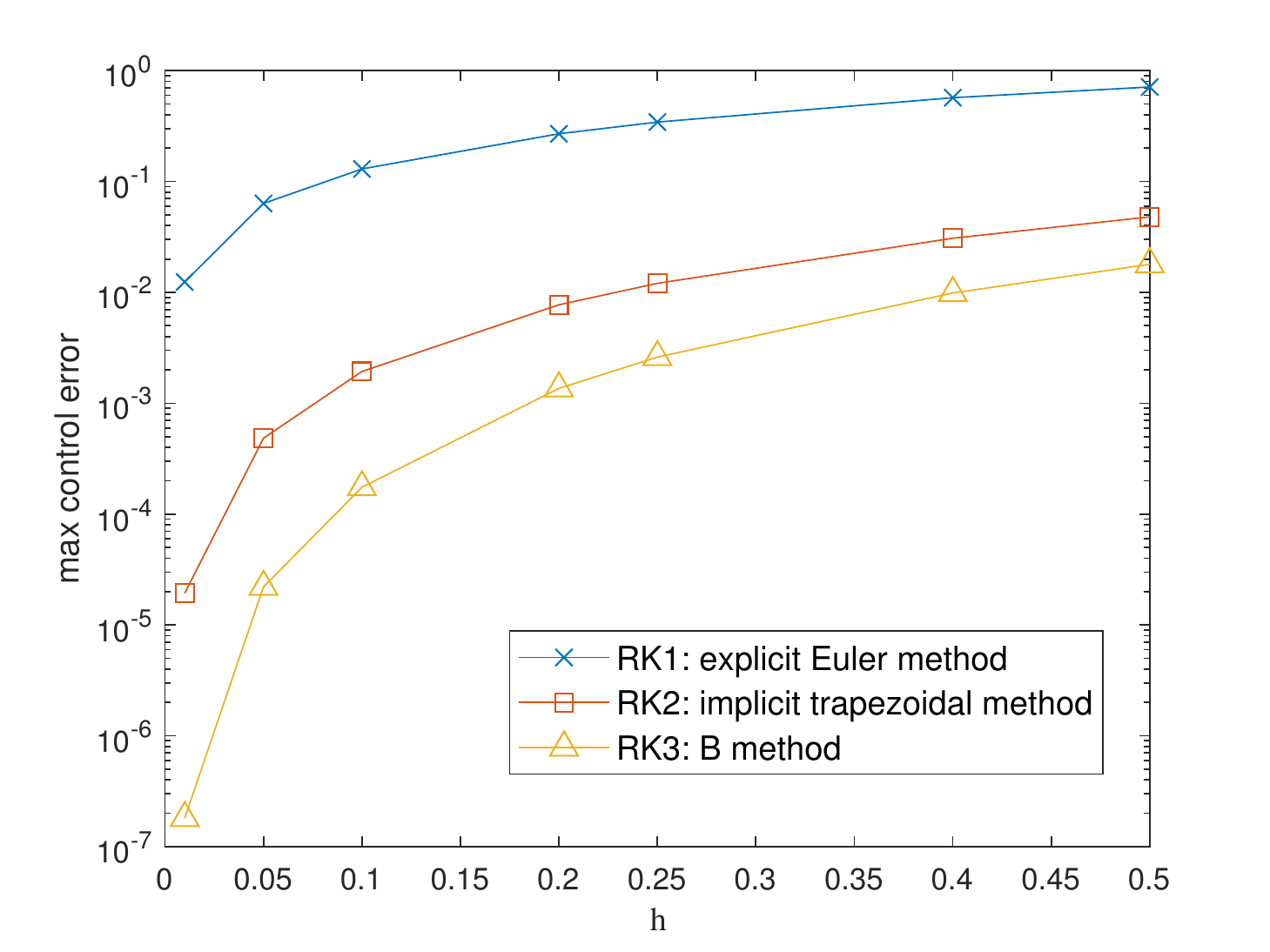}
	\caption{Comparison of different RK discretization: maximum control error $(\text{max} \{ ||u_k-u^*(t_k)||:0\le k \le N\})$ for linear spring oscillator problem.}
\end{figure}
	
\section{ILQR method based on high-order RK discretization for non-linear quadratic problems}\label{section5}
Consider the following non-linear quadratic problem:
\\
\textbf{Problem 5.1}
\begin{equation}\label{5objective}
	\min_{u}J= \min_{u}\int_{0}^{t_f}( \frac{1}{2}x^{T}Qx+\frac{1}{2}u^{T}Ru)dt+ \frac{1}{2}(x(t_f))^{T}Mx(t_f)
\end{equation}
subject to
\begin{equation}\label{5state equation}	
	\dot{x}=f(x,u),\ x(0)=a,
\end{equation}
where $Q\ge 0,\ R> 0,\ M\ge 0$ are symmetric matrices. Applying a $s$-stage RK method with coefficients $a_{ij},b_i$ to discretize \textbf{Problem 5.1}, we can get the following optimization problem:
\\
\textbf{Problem 5.2}
	\begin{equation}\label{5d objective}
	\min_{u_{ki}}J_d=
	\min_{u_{ki}} {\textstyle \sum_{k=0}^{N-1}}h {\textstyle \sum_{i=1}^{s}}b_i(\frac{1}{2}x_{ki}^TQx_{ki}+\frac{1}{2}u_{ki}^TRu_{ki})+\frac{1}{2}x_N^TMx_N      
\end{equation}
subject to
	\begin{subequations}
		\begin{align}
			x_{k+1}&=x_k+h {\textstyle \sum_{i}^{s}}b_if(x_{ki},u _{ki}),\ x_0=a, &0&\le k \le N-1,  \label{5dstate1}\\
			x_{ki}&=x_k+h {\textstyle \sum_{j}^{s}}a_{ij}f(x_{kj},u _{kj}), &0&\le k \le N-1,\ 1\le i \le s. \label{5dstate2}  
		\end{align}
	\end{subequations}
We denote the solutions of \textbf{Problem 5.2} as $u^*_{ki}$ and the corresponding discrete states calculated by \eqref{5dstate1} and \eqref{5dstate2} as $x^*_{ki}$ and $x^*_k$. Different from the linear case in Section \ref{section4}, the value function $V_k(x_k)$ for discrete cost function \eqref{5d objective} is not in a quadratic form in general, and consequently we cannot expect to get a linear feedback law as \eqref{4d linear feedback law}. Besides, note that the third canonical equation \eqref{2 cono3} of \textbf{Problem 5.1} is given by 
\begin{equation}\label{5cano 3}
	D_u^Tf(x,u)p +Ru =0,
\end{equation}
we may be not able to obtain an explicit expression for $u=\hat{u}(x,p)$.   
\par
There are many notations in this section, we list them in advance for convenience. Same as Section \ref{section4}, we also introduce the notations $U_k=(u_{k1},...,u_{ks})$ and $X_k=(x_{k1},...,x_{ks})$ for the vectors combined by internal-stage state and control values during $[t_k,t_{k+1}]$ respectively. Further, we denote the vectors composed by all the internal-stage controls, internal-stage states and node states as $U=(U_0,...,U_{N-1})$, $X=(X_0,...,X_{N-1})$ and $x_d=(x_0,...,x_N)$, respectively. When we add a symbol on these new notations, it just means adding the same symbol to all the elements in brackets. For example, $x_d^0=(x^0_0,...,x^0_N)$. In the rest of this section, we always suppose the time step $h$ in \eqref{5dstate1} and \eqref{5dstate2} is small enough to ensure that, according to implicit function theorem, $X$ and $x_d$ can be uniquely determined by $U$. Note that the total dimension of $(U,X,x_d)$ is $N_{\text{total}}=msN+nsN+n(N+1)$, thus the constraints \eqref{5dstate1} and \eqref{5dstate2} determine a submanifold $\mathcal{M}$ of Euclidean space $\mathbb{R}^{N_{\text{total}}}$: $\mathcal{M}=\left \{ (U,X,x_d)\in \mathbb{R}^{N_{\text{total}}}: \text{\eqref{5dstate1}, \eqref{5dstate2}}\right \}.$

\subsection{Algorithm: ILQR method based on high-order RK discretization}\label{section5.1}
 Before stating the algorithm formally, we give an inspiring example which is helpful to understand the iterative progress in ILQR method.
 
\begin{example}\label{5ILQR example}
Considering the following optimization problem,
\begin{equation}\label{5ex1}
		\min_{u}j(x,u)=\min _{u}\{\frac{1}{2}x^2+\frac{1}{2}u^2  \}  
\end{equation}
subject to $x=g(u)$, where $x,u \in \mathbb{R}$ and $g: \mathbb{R} \to \mathbb{R}$ is twice differentiable. 
\par
As we can see in Figure \ref{ILQR_example}, the function $x=g(u)$ determines a curve, and the minimal for objective function \eqref{5ex1} is just the point on the curve closest to the origin $(0,0)$. First, we choose an initial value $u^0$ and thus determine a point $(u^0,x^0)=(u^0,g(u^0))$ on the curve. Then we can draw the tangent of curve at point $(u^0,x^0)$ and search for minimal $(\tilde{u},\tilde{x})$ of objective function $j$ over this tangent. Once $\tilde{u}$ is determined, we obtain a search direction $\tilde{u}-u^0$ $($arrow direction in Figure \ref{ILQR_example}$)$. By mean of a linear search technique, we can obtain the next iterative point $(u^1,x^1)=(u^1,g(u^1))$ on the curve segment $	\gamma (\alpha)=\{(u(\alpha),x(\alpha )):u(\alpha )=u^0+\alpha (\tilde u-u^0),\ x(\alpha)=g(u(\alpha ))\},\ \alpha \ge 0$. 
\end{example}

\begin{figure}[H]\label{ILQR_example}
	\centering
	\includegraphics[width=0.6\linewidth]{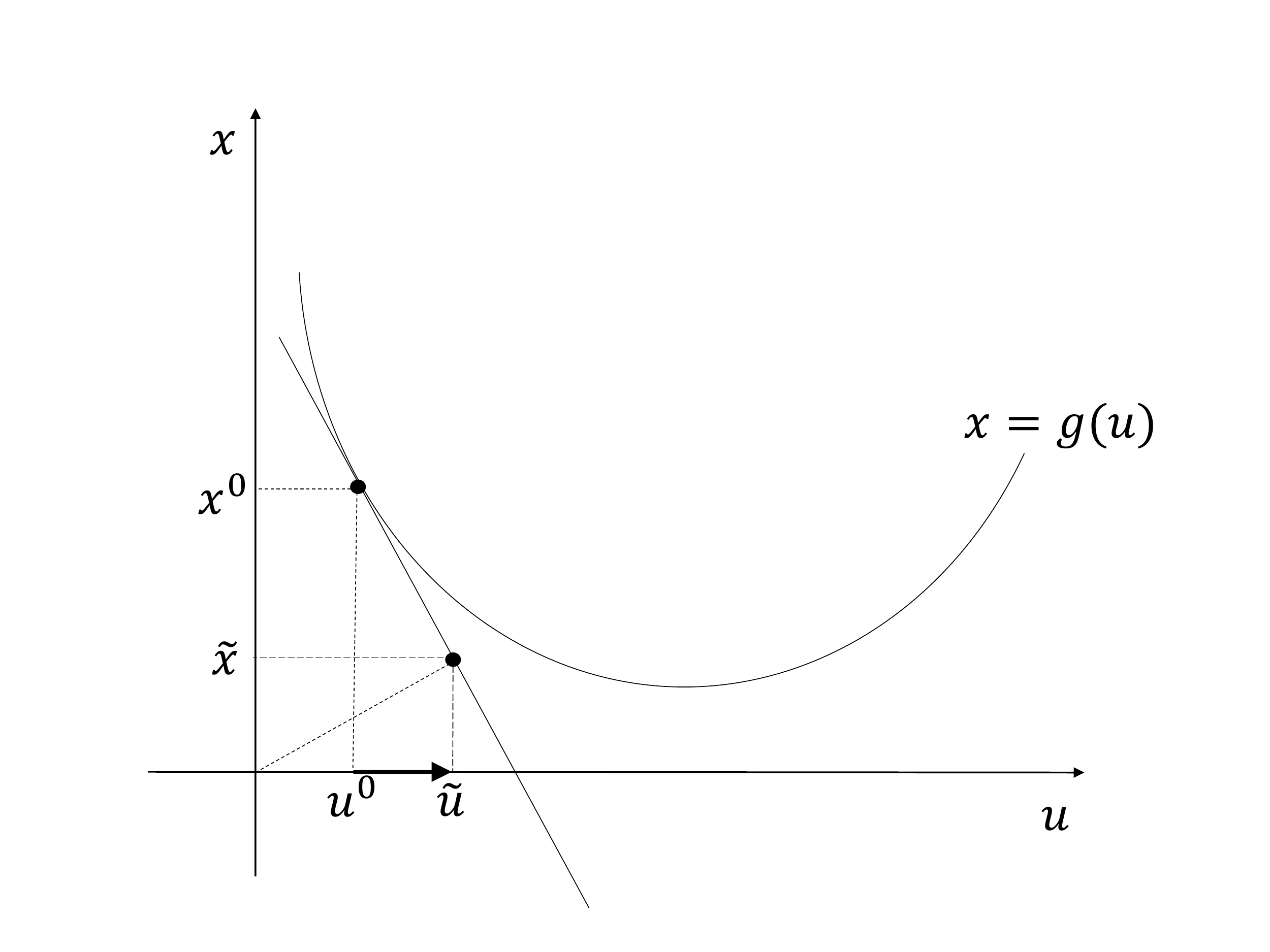}
	\caption{The search direction for Example \ref{ILQR_example}.}
\end{figure}
\par  
Now we turn to the ILQR method which can be divided into two steps. The first step aims at solving \textbf{Problem 5.2} in a similar way as Example \ref{5ILQR example}, and the node control will be calculated in the second step.
\begin{enumerate}[(1)]
	\item \textbf{Iterative progress}. Choosing an initial value for internal-stage control $U^0$, then we can determine a point $(U^0,X^0,x_d^0)$ on $\mathcal{M}$ where $X^0$ and $x_d^0$ are calculated by \eqref{5dstate1} and \eqref{5dstate2}. The equations of tangent plane of $\mathcal{M}$ at this point is given by:
\begin{subequations}
	\begin{align}
		{x}_{k+1}-x^0_{k+1}=&
		{x}_k-x^0_k+h {\textstyle \sum_{i=1}^{s}}b_i(D_xf( x^0_{ki},u^0_{ki}  )( {x}_{ki}-x^0_{ki}) \nonumber \\
		&+D_uf(x^0_{ki},u^0_{ki})({u}_{ki}-u^0_{ki}) ), \label{5tangent plane 1}\\
		{x}_{ki}-x^0_{ki}=&{x}_k-x^0_k+h {\textstyle \sum_{j=1}^{s}}a_{ij}(D_xf(x^0_{kj},u^0_{kj})({x}_{kj}-x^0_{kj})\nonumber \\		&+D_uf(x^0_{kj},u^0_{kj})({u}_{kj}-u^0_{kj}) ), \label{5tangent plane 2}
	\end{align}
\end{subequations}
$0\le k \le N-1,\ 1\le i\le s$ and $x_0^0=a$. Similar to Example \ref{5ILQR example}, we search the minimum $(\tilde{U},\tilde{X},\tilde{x}_d)$ of cost function \eqref{5d objective} over this tangent plane. The progress to obtain $\tilde{U}$ is given in Appendix \ref{App1}. With $\tilde{U}$ calculated, we get a search direction $\tilde{U}-U^0$ for \textbf{Problem 5.2}. Considering the curve segment on $\mathcal{M}$:
\begin{equation}\label{5curve segment}
	\begin{aligned}
		\gamma(\alpha)= \{ (U,X,x_d):U=U^0+\alpha (\tilde{U}-U^0),\\
		\text{$X$ and $x_d$ are calculated by \eqref{5dstate1} and \eqref{5dstate2} } 
		\} ,\ \alpha \ge 0,
	\end{aligned}
\end{equation}
we can choose the step length $\alpha$ by some linear search technique and then let $\gamma(\alpha)$ be the next iterative point $(U^1,X^1,x^1_d)$ on $\mathcal{M}$. Repeat this progress until a preset condition for interruption is satisfied, then we can take the finial iterative point (suppose it is the $L$-th point) as the approximate solution of \textbf{problem 5.2}: $(U^L,X^L,x^L_d)\approx (U^*,X^*,x_d^*)$.   
\item \textbf{Calculation of node controls}.  
According to Proposition \ref{2 key proposition}, the solution of \textbf{problem 5.2} $u^*_{ki},\ x^*_{ki}$ and $x^*_k$ will satisfy the following equations  
	\begin{subequations}
		\begin{align}
			p^*_{k+1}&=p^*_k-h {\textstyle \sum_{i=1}^{s}}b_i( D_x^Tf(x^*_{ki},u^*_{ki})p^*_{ki}+  Qx^*_{ki}  ),	\label{5ad1}\\
			p^*_{ki}&=p^*_k-h {\textstyle \sum_{i=1}^{s}}\bar{a}_{ij} ( D_x^Tf(x^*_{kj},u^*_{kj})p^*_{kj}+  Qx^*_{kj}  ),\label{5ad2}
		\end{align}
	\end{subequations} 
$0\le k\le N-1,\ 1\le i \le s $, with boundary condition $p^*_N=Mx^*_N$. Since we have obtained $u^*_{ki},\ x^*_{ki}$ and $x^*_k$ in step (1), we can calculate the discrete costates $p^*_{N-1},...,p^*_{0}$ recursively by solving \eqref{5ad1} and \eqref{5ad2}. Plugging $ x^*_k$ and $p^*_k$ into the third canonical equation \eqref{5cano 3}, we can obtain the node control $u^*_k$ by solving   
	\begin{equation}\label{5 control law}\nonumber
		D_u^Tf(x_k^*,u_k^*)p_k^*+Ru_k^*=0,\ 0\le k \le N.
	\end{equation}
	\end{enumerate}

\subsection{Convergence analysis}
In this subsection, we give a convergence analysis for the iterative progress of ILQR method. The main results contain:
\begin{enumerate}[(i)]
	\item
	The iterative progress of ILQR technique belongs to the quasi-Newton type methods and produces descent search direction for \textbf{Problem 5.2}.
	\item 
	ILQR cannot attain super-linear convergence rate in general.   
\end{enumerate} 
To begin with, let us rewritten the cost function \eqref{5d objective} in the following form: 
\begin{equation} \label{5c objective11}
\min_{U} J_d(U,X,x_d)=\min_{U}\frac{1}{2}X^T\mathcal{Q}X+\frac{1}{2}U^T\mathcal{R}U+\frac{1}{2} x_N^TMx_N,  
\end{equation}
where 
\begin{equation}\nonumber
	\mathcal{Q}=\begin{bmatrix}
		Q_h(1) &  &0 \\
		  & \ddots  &  \\
		0 &   &Q_h(N)
	\end{bmatrix},\ \mathcal{R}=\begin{bmatrix}
		R_h(1) &  &0 \\
	  & \ddots  &  \\
		0 &  &R_h(N)
	\end{bmatrix} ,
\end{equation}
 $Q_h$ and $R_h$ are the same as \eqref{Q_h R_h}.
\par
Since the internal-stage state $X$ and node states $x_k\ (k=0,...,N)$ are determined by the internal-stage control $U$ according to equations \eqref{5dstate1} and \eqref{5dstate2}, it is convenient to introduce the notations $X=F(U)$ and $x_k=F_d^k(U)\ (k=0,...,N)$. Specially, since $x_0=a$, we have $F_d^0(U)\equiv a$. Note that when we plug these functions in \eqref{5c objective11}, we can express the discrete cost function with only variable $U$, namely, 
\begin{equation}
	J_d(U)=\frac{1}{2}(F(U))^T\mathcal{Q}F(U)+\frac{1}{2}U^T\mathcal{R}U+\frac{1}{2} (F_d^N (U))^TM F_d^N (U). 
\end{equation}
For convenience of the discussion latter, we give the gradient and Hesse matrix of $J_d(U)$ at $U^0$:
\begin{align}
	(J_d)'(U^0)= &(F'(U^0))^T\mathcal{Q} F(U^0)+\mathcal{R}U^0+((F_d^N)'(U^0))^TMF_d^N(U^0) \label{5c gradient},\\	
	(J_d)''(U^0)=&(F'(U^0))^T\mathcal{Q} F'(U^0)+(F''(U^0))^T\mathcal{Q} F(U^0)+\mathcal{R} \nonumber \\
	&+((F_d^N)'(U^0))^TM (F_d^N)'(U^0)+((F_d^N)''(U^0))^TM F_d^N(U^0) \label{5c Hesse}. 
\end{align}
Next we will calculate the search direction generated by ILQR iterative progress. The tangent plane of manifold $\mathcal{M}$ at the point $(U^0,X^0,x^0_1,...,x^0_N)$ $=(U^0,F(U^0)$ $,F_d^1(U^0),...,F_d^N(U^0))$ is given by
\begin{align}
	X-X^0&=F'(U^0) (U-U^0), \label{5c 1} \\
x_k-x_k^0 &= (F_d^k)'(U^0)(U-U^0),\ 0\le k \le N. \label{5c 2}
\end{align}
Particularly, the $(N+1)$-th formula of \eqref{5c 2} is 
\begin{equation}\label{5c x_N}
x_N-x_N^0=(F_d^N)'(U^0) (U-U^0).
\end{equation}

Substituting \eqref{5c 1} and \eqref{5c x_N} into \eqref{5c objective11}, we can obtain the following expression after combing similar terms:
\begin{equation}\label{5c Jd}
\frac{1}{2}(U-U^0)^TW(U^0)(U-U^0)+(U-U^0)^TY(U^0)+C(U^0) 
\end{equation}
where
\begin{subequations}
\begin{align}
	W(U^0)=&(F'(U^0))^T\mathcal{Q} F'(U^0)+\mathcal{R}+((F_d^N)'(U^0))^TM (F_d^N)'(U^0) , \label{5c W}\\
	Y(U^0)= &(F'(U^0))^T\mathcal{Q}X^0+\mathcal{R}U^0+((F_d^N)'(U^0))^TMx_N^0 ,\label{5c Y}\\
	C(U^0)=&\frac{1}{2}(X^0)^T\mathcal{Q}X^0+\frac{1}{2}(U^0)^T\mathcal{R}U^0+\frac{1}{2}(x_N^0)^TMx_N^0 .\label{5c C}    
\end{align}	
\end{subequations}
Since $W(U^0)$ is a symmetric matrix, the necessary condition for minimizing \eqref{5c Jd} is
\begin{equation}\nonumber
	W(U^0)(U-U^0)+Y(U^0)=0,
\end{equation} 
then the search direction of ILQR can be expressed as 
\begin{equation}\label{5c iLQR direction}
	U-U^0=-( W(U^0)  )^{-1} Y(U^0)	.
\end{equation}
Comparing \eqref{5c gradient} with \eqref{5c Y}, we have $Y(U^0)=(J_d)'(U^0)$. Consequently, we can rewrite the ILQR search direction \eqref{5c iLQR direction} in a quasi-Newton type 
\begin{equation}
	U-U^0=-(W(U^0))^{-1}(J_d)'(U^0),
\end{equation}
According to \eqref{5c W}, $W(U^0)$ is positive-definite, thus the search direction generated by ILQR is a decrease one for \textbf{Problem 5.2}.
\par
Now we study whether the convergence rate of ILQR is super-linear. Suppose the sequence $\{U^l\}$, generated by ILQR iteration $U^{l+1}=U^l-W(U^l)^{-1} (J_d)'(U^l)$, converges to the minimum of \textbf{Problem 5.2} (denoted by $U^*$). According to the necessary and sufficient condition for super-linear convergence (see Theorem 3.7 in \cite{Nocedal1999NumericalO}), ILQR can attain a super-linear convergence rate if and only if 
\begin{equation}\label{5c superlinear}
	\lim_{l \to +\infty} \frac{||(W(U^l) -(J_d)''(U^l))\cdot 	(U^{l+1}-U^l) ||}{|| 	U^{l+1}-U^l||} =0.
\end{equation}
However, \eqref{5c superlinear} cannot be achieved in general. To demonstrate this, let us go back to Example \ref{5ILQR example} where $W(u^l)=1+(g'(u^l) )^2$ and $j''(u^l)=1+(g'(u^l))^2+g''(u^l)g(u^l)$. We have
\begin{equation}\nonumber
	\lim_{l \to +\infty} \frac{|(j''(u^l)-W(u^l))\cdot (u^{l+1}-u^l)|}{|u^{l+1}-u^l|} =|j''(u^*)-W(u^*)|=|g''(u^*)g(u^*) |.
\end{equation}
It is easy to cite an example in which $|g''(u^*)g(u^*) |\ne 0$. When $g(u)=u^2+1$, we have $u^*=0,\ g(u^*)=1$ and $g''(u^*)=2$.
\par

\subsection{Numerical example: inverted pendulum}
Consider the state equations of inverted pendulum under control:
\begin{subequations}
	\begin{align}
		\dot{\theta }&=\omega,\ \theta(0)=\frac{\pi}{3},  \nonumber\\
		\dot{\omega }&=sin\theta +u,\ \omega(0)=0. \nonumber   
	\end{align}
\end{subequations}	
The cost function of this problem is:
\begin{equation}\nonumber
	J=\int_{0}^{4}0.025u^2dt+2.5 x_f^Tx_f,    
\end{equation}	
where $x_f=(\theta(4),\omega(4))$.
\par
For this non-linear problem, we explore the dependence of initial control $u_0$ (the first node control value calculated by ILQR method) on the node number for different RK discretization. As we can see in Figure \ref{fig ILQR}, for the B method with third order in OCP, the calculated initial control $u_0$ converges faster than the other two produced by lower order discretization. When the node number is increased to nearly $800$, the initial control $u_0$ generated by the second order implicit trapezoidal method can attain a satisfying accuracy, while explicit Euler method still behaves badly.      
\begin{figure}[t]
	\centering
	\includegraphics[width=0.6\linewidth]{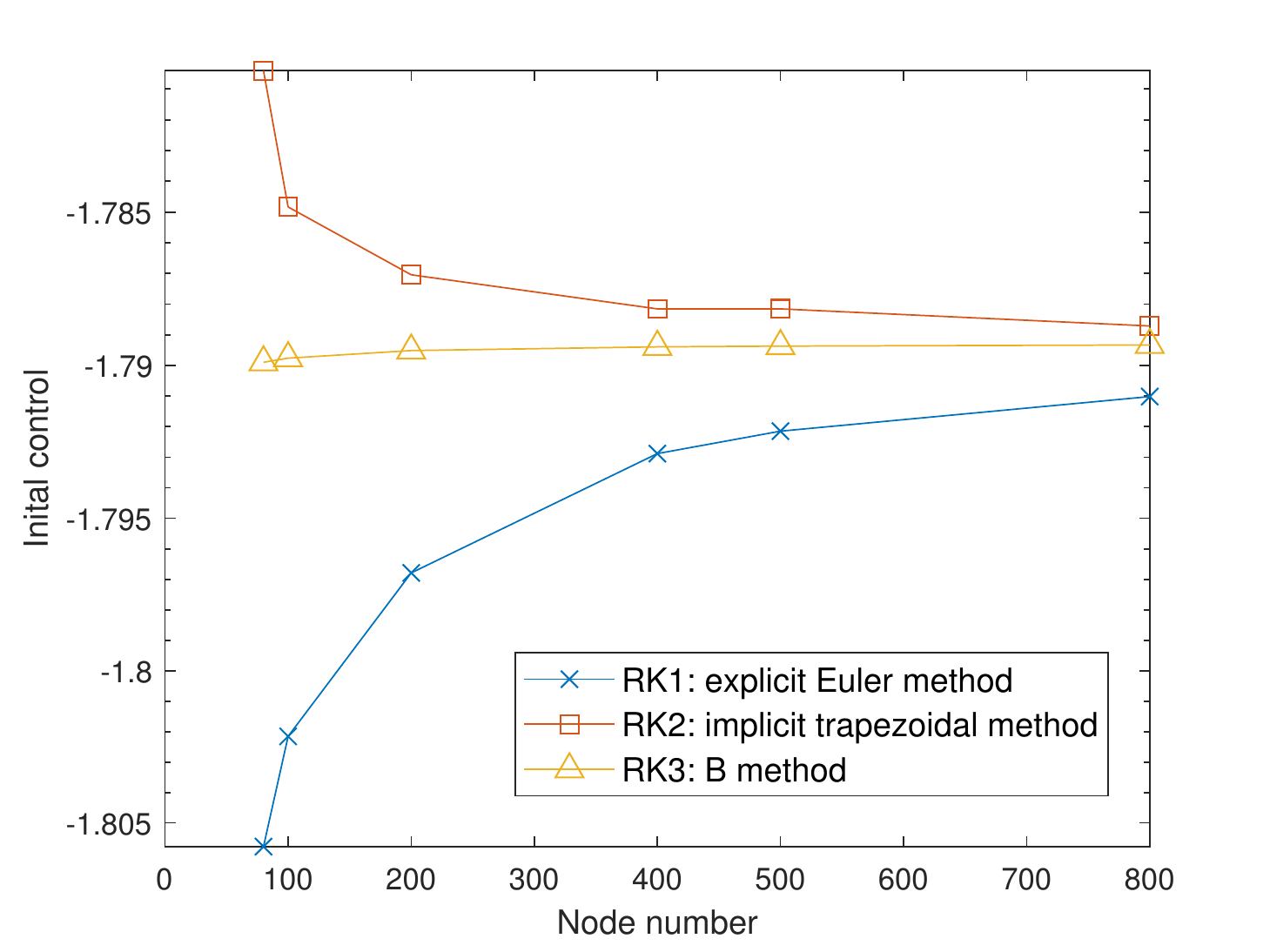}
	\caption{Comparison of different RK discretization: dependence of initial control on the node number.}
	\label{fig ILQR}
\end{figure}

\section{Conclusion}
We propose DLQR and ILQR methods based on high-order RK discretization to solve linear and nonlinear quadratic optimal control problems respectively. The equivalence between direct approach based on RK discretization and SPRK integration for canonical equations, is reconstructed in the framework of dynamic programming and plays important roles in both theoretical analysis and algorithm implementation. Based on this equivalence, we deduce the order conditions for internal-stage controls in Theorem \ref{3internal-stage control order conditions} and discover that some common RK methods with high order in OCP (third or fourth order explicit methods; implicit methods such as implicit trapezoidal method) cannot avoid the order reduction. Then, by mean of the the equivalence in Section \ref{section2}, we calculate node control in DLQR and ILQR methods to overcome this obstacle. Numerical examples show the advantage of our methods on the accuracy increase. Because both DLQR and ILQR methods are achieved with the feedback technique, they have high computational efficiency. In the convergence analysis for ILQR iteration, we demonstrate that ILQR technique is essentially a quasi-Newton method and the convergence rate is linear in general.  

\appendix
\section{Proof for proposition \ref{2 key proposition}}\label{App for section2}
Denoting the gradient of discrete value function $(V_{k})'(x_k)$ as $p_k$ $(0\le k\le N)$, we have $p_N=(V_{N})'(x_N)=\Phi'(x_N)$ according to the boundary condition $V_N(x_N)=\Phi(x_N)$. To achieve the minimization of the right hand of discrete HJB equation \eqref{2d HJB}, we first introduce a multiply $\varphi_{ki}$ for the $i$-th intermediate equation of \eqref{2dstate1} $(1\le i \le s)$ and define
\begin{equation}\nonumber
	\begin{aligned}
		\mathcal{V}_k&(x_k,x_{k1},...,x_{ks},u_{k1},...,u_{ks},\varphi _{k1},...,\varphi _{ks}) \\
		=&h {\textstyle \sum_{i=1}^{s}}b_iC(x_{ki},u_{ki})+V_{k+1}(x_k+ h{\textstyle \sum_{i=1}^{s}} b_if(x_{ki},u_{ki}))\\
		&+{\textstyle \sum_{i=1}^{s}} \varphi _{ki}^T(x_k+h {\textstyle \sum_{j=1}^{s}}a_{ij}f(x_{kj},u_{kj})-x_{ki}) ,
	\end{aligned}
\end{equation}
then the necessary conditions for minimization state that the partial derivative of $\mathcal{V}_k$ with respect to $x_{ki}$, $u_{ki}$ and $\varphi _{ki}$ $(i=1,...,s)$ should vanish. Specifically, we have,
\begin{subequations}
	\begin{align}
		&D_{x_{ki}} \mathcal{V}_k(x_k,x_{k1},...,x_{ks},u_{k1},...,u_{ks},\varphi _{k1},...,\varphi _{ks})\nonumber \\
		&=
		hb_i [D_xC(x_{ki},u_{ki})+D_x^Tf(x_{k_i},u_{ki})(p_{k+1}+ {\textstyle \sum_{j=1}^{s}}\frac{a_{ji}}{b_i}\varphi_{kj})]-\varphi_{ki} =0 ,
		\label{2H1} \\
		&D_{u_{ki}} \mathcal{V}_k(x_k,x_{k1},...,x_{ks},u_{k1},...,u_{ks},\varphi _{k1},...,\varphi _{ks})
		\nonumber	
		\\
		&=hb_i [D_uC(x_{ki},u_{ki})+D_u^Tf(x_{k_i},u_{ki})(p_{k+1}+ {\textstyle \sum_{j=1}^{s}}\frac{a_{ji}}{b_i}\varphi_{kj})]=0 ,  \label{2H2} \\
		&D_{\varphi_{ki}} \mathcal{V}_k(x_k,x_{k1},...,x_{ks},u_{k1},...,u_{ks},\varphi _{k1},...,\varphi _{ks})
		\nonumber	
		\\
		&=x_k+h {\textstyle \sum_{j=1}^{s}}a_{ij}f(x_{kj},u_{kj})-x_{ki}=0 \label{2H3}
	\end{align}
\end{subequations}
for all $1\le i \le s$. The discrete HJB equation \eqref{2d HJB} can now be rewritten as
\begin{equation} \label{A1 dHJB}
	\begin{aligned}
	V_k(x_k)=&\mathcal{V}_k(x_k,x_{k1},...,x_{ks},u_{k1},...,u_{ks},\varphi _{k1},...,\varphi _{ks}) \\
=&h {\textstyle \sum_{i=1}^{s}}b_iC(x_{ki},u_{ki})+V_{k+1}(x_k+ h{\textstyle \sum_{i=1}^{s}} b_if(x_{ki},u_{ki}))\\
&+{\textstyle \sum_{i=1}^{s}} \varphi _{ki}^T(x_k+h {\textstyle \sum_{j=1}^{s}}a_{ij}f(x_{kj},u_{kj})-x_{ki})
	\end{aligned}
\end{equation} 
in which $x_{ki},\ u_{ki}$ and $\varphi _{ki}$ $(1\le i\le s)$ are functions with variable $x_k$ implicitly determined by \eqref{2H1}, \eqref{2H2} and \eqref{2H3}. Derivate both side of \eqref{A1 dHJB} with respect to $x_k$ and note that the partial derivatives of $\mathcal{V}$ with respect to $x_{ki}$, $u_{ki}$, $\varphi_{ki}$ all vanish, we can obtain:
\begin{equation}\label{pk pk+1}
	p_k= p_{k+1} + {\textstyle \sum_{i=1}^{s}}\varphi _{ki} .
\end{equation}   
By introducing the variables
\begin{equation}\label{2pki}
	p_{ki} \triangleq p_{k+1}+ {\textstyle \sum_{j=1}^{s}}\frac{a_{ji}}{b_i}\varphi_{kj},\ 1\le i \le s,
\end{equation}
we can rewrite the conditions \eqref{2H1} and \eqref{2H2} as 
\begin{subequations}
	\begin{align}
		&hb_i (D_xC(x_{ki},u_{ki})+D_x^Tf(x_{ki},u_{ki}) p_{ki})=\varphi_{ki}, \label{2H4}
		\\
		&hb_i (D_uC(x_{ki},u_{ki})+D_u^Tf(x_{ki},u_{ki}) p_{ki})=0.  \label{2H5}
	\end{align}
\end{subequations} 
Since $h>0,\ b_i>0$, \eqref{2H5} is equivalent to the following formula
\begin{equation}\nonumber
D_uC(x_{ki},u_{ki})+D_u^Tf(x_{ki},u_{ki}) p_{ki}=0
\end{equation}
which has the same form as the third canonical equation \eqref{2 cono3}, thus we have $u_{ki}=\hat{u}(x_{ki},p_{ki})$. Substituting \eqref{2H4} into \eqref{pk pk+1}, we can obtain
\begin{equation}\label{2db1}
	p_k=p_{k+1}+h {\textstyle \sum_{k=1}^{s}}b_i(D_xC(x_{ki},u_{ki})+D_x^Tf(x_{ki},u_{ki})p_{ki}).
\end{equation}
Plugging \eqref{2H4} and \eqref{2db1} into \eqref{2pki}, we can get the expression of $p_{ki}$, namely,
\begin{equation} \nonumber
	p_{ki}=p_k-h {\textstyle \sum_{j=1}^{s}}\bar{a}_{ij}[D_x^Tf(x_{kj},u_{kj})p_{kj}+D_xC(x_{kj},u_{kj})] ,\ 1\le i \le s,
\end{equation}
where $\bar{a}_{ij}=b_j-\frac{b_ja_{ji}}{b_i}$.
\par
Now we have proven that the necessary condition for \textbf{Problem 2.2} is given by the equations \eqref{2 ad1}-\eqref{2 ad4} in Proposition \ref{2 key proposition}. The existence and uniqueness of solution can be ensured when time step $h$ is small enough (more details can be seen in \cite{Hager2000RungeKuttaMI}).

\section{Calculation of search direction in ILQR method}\label{App1}
The search direction $\tilde U -U^0$ of ILQR in Subsection \ref{section5.1} can be calculated by solving the following optimization problem in which the minimum is $\tilde U$. 
\begin{equation}\nonumber
	\min_{U_k}  {\textstyle \sum_{k=0}^{N-1}}(\frac{1}{2}X_k^TQ_hX_k+ 
	\frac{1}{2} U_k^TR_hU_k    )+\frac{1}{2}x_N^TMx_N
\end{equation}
subject to
\begin{subequations}
	\begin{align}
x_{k+1}-x^0_{k+1}&=x_k-x^0_k+B_k(X_k-X_k^0)+C_k(U_k-U_k^0), \label{A1}\\
X_{k}-X^0_{k}&=Z(x_k-x_k^0)+A^1_k(X_k-X_k^0)+A^2_k(U_k-U_k^0), \label{A2}
	\end{align}
\end{subequations}
where $Q_h$ and $R_h$ have the same definitions as in \eqref{4d obj}, and
\begin{subequations}
	\begin{align}
A^1_k&=h\begin{bmatrix}
	a_{11}D_xf(x^0_{k1},u^0_{k1})& \cdots &a_{1s}D_xf(x^0_{ks},u^0_{ks}) \\
	\vdots &  &\vdots \\
	a_{s1}D_xf(x^0_{k1},u^0_{k1}) & \cdots &a_{ss}D_xf(x^0_{ks},u^0_{ks})
\end{bmatrix}, \nonumber \\
	A^2_k&=h\begin{bmatrix}
	a_{11}D_uf(x^0_{k1},u^0_{k1})& \cdots &a_{1s}D_uf(x^0_{ks},u^0_{ks}) \\
	\vdots &  &\vdots \\
	a_{s1}D_uf(x^0_{k1},u^0_{k1}) & ... &a_{ss}D_uf(x^0_{ks},u^0_{ks})
\end{bmatrix},
Z=\begin{bmatrix}
	I_n (1)\\
	\vdots \\
	I_n(s)
\end{bmatrix}, \nonumber	
	\end{align}
\end{subequations}

\begin{equation}\nonumber
	\begin{aligned}
	B_k&=h\begin{bmatrix}
	b_1D_xf(x^0_{k1},u^0_{k1}) &...  &b_sD_xf(x^0_{ks},u^0_{ks}) 
\end{bmatrix}, \\
C_k&=h\begin{bmatrix}
	b_1D_uf(x^0_{k1},u^0_{k1}) &...  &b_sD_uf(x^0_{ks},u^0_{ks}) 
\end{bmatrix}.
	\end{aligned}
\end{equation}
We mention that the constraints \eqref{A1} and \eqref{A2} are the same as the tangent plane equations \eqref{5tangent plane 1} and \eqref{5tangent plane 2}. After some simple operation, we can transform the constraints \eqref{A2} and \eqref{A1} into the form
\begin{subequations}
	\begin{align}
	X_k=E_k x_k +F_k U_k +D_k^1, \label{A3} \\
	x_{k+1}= G_k x_k +H_k U_k +D_k^2, \label{A4}
	\end{align}
\end{subequations} 
where
\begin{subequations}
\begin{align}
	E_k&=(I_{sn}-A_k^1)^{-1}Z, &F_k&=(I_{sn}-A_k^1)^{-1}A_k^2, \nonumber \\
	G_k&=I_n+B_kE_k,\ &H_k&=B_kF_k+C_k,\nonumber\\
	D_k^1&=X_k^0-E_kx_k^0-F_kU_k^0,\ &D_k^2&=x_{k+1}^0-G_kx_k^0-H_kU_k^0.\nonumber
\end{align}
\end{subequations}
Now we can write the discrete HJB equation as following
\begin{equation} \label{A HJB}
	\begin{aligned}
V_k(x_k)=&\min_{U_k}\{\frac{1}{2}X_k^TQ_hX_k+\frac{1}{2}U_k^TR_hU_k+V_{k+1}(x_{k+1})\} \\
=  &\min_{U_k} \{ \frac{1}{2}(E_kx_k+F_kU_k+D_k^1)^TQ_h (E_kx_k+F_kU_k+D_k^1)\\
&+\frac{1}{2}U_k^TR_hU_k+V_{k+1}(G_kx_k+H_kU_k+D_k^2)\},\ 0\le k \le N-1,
	\end{aligned}
\end{equation}
with boundary condition $V_N(x_N)=\frac{1}{2}x_N^TMx_N$. Then, we suppose $V_k(x_k)$ has a quadratic form
\begin{equation}\label{A vk}
	V_k(x_k)=\frac{1}{2}x_k^TM_kx_k+Y_k^Tx_k+c_k,\ 0\le k \le N, 
\end{equation}
where $M_k$ is symmetric. According to the boundary condition, we have $M_N=M$, $Y_N=0$, $c_N=0$. Substituting \eqref{A vk} into the discrete HJB equation \eqref{A HJB} and let the derivative of the right hand of equality with respect to $U_k$ vanish, we can get the following linear feedback law after combing like terms and matrix inversion
\begin{equation}\label{A feedback law}
	U_k=U_k^1x_k+U_k^2
\end{equation}
where
\begin{subequations}
	\begin{align}
U_k^1&=-(F_k^TQ_hF_k+R_h+H_k^TM_{k+1}H_k)^{-1}(F_k^TQ_hE_k+H_k^TM_{k+1}G_k), \label{A uk1} \\
U_k^2&=-(F_k^TQ_hF_k+R_h+H_k^TM_{k+1}H_k)^{-1}(F_k^TQ_hD_k^1+H_k^TM_{k+1}D_k^2+H_k^TY_{k+1}),\label{A uk2}
	\end{align}
\end{subequations}
$0\le k \le N-1$. Plugging \eqref{A feedback law} into the discrete HJB equation \eqref{A HJB}, we can obtain the following iterative formulas for $M_k$ and $Y_k$ by comparing both sides of equality:
\begin{subequations}
\begin{align}
	M_k=&(E_k+F_kU_k^1)^TQ_h(E_k+F_kU_k^1)+(U_k^1)^TR_hU_k^1
	\nonumber \\
	&+(G_k+H_kU_k^1)^TM_{k+1}(G_k+H_kU_k^1), \nonumber \\
	Y_k=&(E_k+F_kU_k^1)^TQ_h(F_kU_k^2+D_k^1)+(U_k^1)^TR_hU_k^2
	\nonumber \\
	&+(G_k+H_kU_k^1)^TM_{k+1}(H_kU_k^2+D_k^2)+(G_k+H_kU_k^1)^TY_{k+1}, \nonumber
\end{align}
\end{subequations}   
$0\le k \le N-1$. With $M_N$ and $Y_N$ known, we can calculate $M_k$ and $Y_k$ $(k=N-1,...,0)$ recursively. Then according to \eqref{A uk1} and \eqref{A uk2}, we can obtain the coefficient matrices $U_k^1$ and $U_k^2$. Plugging the feedback law \eqref{A feedback law} into \eqref{A4}, the discrete state trajectory can be obtained by 
\begin{equation}\nonumber
	x_{k+1}=(G_k+H_kU_k^1)x_k+H_kU_k^2+D_k^2,\ x_0=a,\ 0\le k\le N-1.
\end{equation}
At last, the internal stage control $U_k$ can be calculated by substituting $x_k$ into \eqref{A feedback law}.

\bibliographystyle{siamplain}
\bibliography{references}
\end{document}